\documentclass[a4paper,10pt]{amsart}
\usepackage{amsmath,amsthm,amsfonts,amssymb,calrsfs}

\usepackage[all]{xy}
\usepackage{tikz-cd, tikz}

\usepackage{graphicx}
\usepackage{yfonts}
\usepackage{xcolor}

\usepackage{array}
\usepackage{multirow}
\usepackage{caption}

\usepackage{todonotes}

\usepackage[hyperindex=true, colorlinks=true, linkcolor=blue, filecolor=black,
citecolor=blue, urlcolor=blue,
pagebackref=true,
bookmarks=true, bookmarksopen=true, bookmarksnumbered=true,
bookmarksopenlevel=2, pdfborder={0 0 0}]{hyperref}
\usepackage{cleveref}

\theoremstyle{plain}{
    \newtheorem{theorem}{Theorem}[section]
    \newtheorem{lemma}[theorem]{Lemma}
    \newtheorem{corollary}[theorem]{Corollary}
    \newtheorem{proposition}[theorem]{Proposition}
    
}

\theoremstyle{definition}{
    \newtheorem{definition}[theorem]{Definition}
    \newtheorem{defn-thm}[theorem]{Definition-Theorem}
    \newtheorem{example}[theorem]{Example}
    
}

\theoremstyle{remark}{
    \newtheorem{remark}[theorem]{Remark}

}

\def\rank{\mathrm{rank}}

\def\um{{\bf m}}
\def\uc{{\bf c}}

\def\cF{\mathcal{F}}
\def\cE{\mathcal{E}}
\def\cO{\mathcal{O}}

\def\cG{\mathcal{G}}

\def\cK{\mathcal{K}}

\def\cQ{\mathcal{Q}}

\def\R{\mathbb{R}}
\def\Q{\mathbb{Q}}
\def\Z{\mathbb{Z}}
\def\N{\mathbb{N}}

\def\P{\mathbb{P}}

\def\C{\mathbb{C}}

\def\>{\rangle}

\def\<{\langle}
\def\>{\rangle}

\def\Hom{\mathrm{Hom}}
\def\Spec{\mathrm{Spec}}

\def\codim{\mathrm{codim}}  
\def\Supp{\mathrm{Supp}}
\def\Gr{\mathrm{Gr}}

\def\up{\underline{p}}

\begin{document}

\title[Stable Toric sheaves. I]{Stable toric sheaves. I \\
Chern classes}

\address{Univ Brest, UMR CNRS 6205, Laboratoire de Mathématiques de Bretagne
Atlantique, France}

\author[C. Tipler]{Carl TIPLER}
\email{carl.tipler@univ-brest.fr}


\begin{abstract}
We reduce Hartshorne's conjecture on indecomposable rank $2$ vector bundles in the semistable case to the non-existence of smoothable rank $2$ semistable toric sheaves on the projective space. We then study Chern classes for rank $2$ torsion-free toric sheaves. For the reflexive ones, we derive a simple formula for the Chern polynomial, and in the general torsion-free case we introduce an iterative construction method based on elementary injections, allowing us to prescribe Chern classes. This yields infinite families of explicit examples on $\mathbb{P}^4$ and $\mathbb{P}^5$, and establishes existence on $\mathbb{P}^n$ for all $n \geq 3$, with Chern classes satisfying all known constraints arising from locally freeness and indecomposability. We also provide simple obstructions for smoothability.
\end{abstract}

\maketitle

\section{Introduction}

The existence of indecomposable low rank vector bundles on complex projective spaces is a longstanding problem in algebraic geometry \cite{Har79,OSS}. While in low dimension there are many constructions \cite{OSS}, in rank $2$ and for $n\geq 4$, there is an essentially unique one : the Horrocks--Mumford bundle on 
 $\C\P^4$ \cite{HoMu} (by ``essentially unique'' we mean that it comes in moduli, and can be twisted by line bundles or pulled back along finite covers). Given the close relation between  codimension two subvarieties and rank two vector bundles,  Hartshorne conjectured that for $n\geq 7$ all such bundles should split \cite[Conjecture 6.3]{Har74}. He also raised the problems of whether examples might exist on $\C\P^5$  or new ones could be found on $\C\P^4$ (\cite[Problems 1 and 2]{Har79}). While these questions remain open in general, substantial progress has been made in the equivariant setting. Viewing $\C\P^n$ as a toric variety, Kaneyama \cite{Kan} and Klyachko \cite{Kly90} proved that any torus-equivariant vector bundle of rank $r<n$ over $\C\P^n$ splits as a direct sum of line bundles. More recently, Ilten and S\"uss extended this splitting result to bundles equivariant under a torus action of lower rank \cite{IlSu}. In this paper, we propose a new approach to this construction problem, through toric methods.

Our starting point is the observation that the existence of a semistable rank $2$ vector bundle is equivalent to that of a smoothable, torus-equivariant, semistable torsion-free sheaf with the same Chern classes (cf. Section \ref{sec:new approach}, Corollary \ref{cor:existence reduction projective space}). This suggests a strategy for approaching Hartshorne's conjecture via the study of deformations of torus-equivariant sheaves. Conversely, potential counterexamples to the conjecture may also arise as deformations of such sheaves. For this program to succeed, one must begin with a torsion-free sheaf with the desired Chern classes, subject to the same constraints as those of locally free ones. As a first step in this direction, we establish an inductive method for constructing torus-equivariant torsion-free sheaves with prescribed Chern classes.

We will denote by $(\uc_1,\ldots,\uc_n)\in \Z^n$ the Chern classes of a torsion-free sheaf $\cE$ where we used an isomorphism $H^{2k}(\P^n,\Z)\simeq \Z\cdot H^k$, where $H$ is (the Poincar\'e dual of) some hyperplane section, so that
$$
H^{2\bullet}(\P^n,\Z)\simeq \Z[H]/\langle H^{n+1}\rangle
$$
and 
$$
\forall k\in\lbrace 1,\ldots,n\rbrace,\: \uc_k(\cE)=\uc_k\cdot H^k.
$$
As a warm up, we first study torus-equivariant reflexive sheaves. Up to normalisation obtained by tensoring with a line bundle (see Section \ref{sec:normal forms}), by work of Perling \cite{Per04}, a rank $2$ reflexive toric sheaf  is entirely characterised by a couple $(c_\rho,L_\rho)_{\rho\in\Sigma(1)}$, where $\Sigma(1)$ denotes the set of rays of the fan of $\C\P^n$ viewed as a toric variety, $(c_\rho)_{\rho\in\Sigma(1)}\in\N^{n+1}$, and $L_\rho\subset \C^2$ is a line for each $\rho\in\Sigma(1)$. With the help of Perling's resolution \cite{Per03}, and our choice of normalisation, we easily obtain :
\begin{proposition}[Corollary \ref{cor:chern classes normalized}]
 \label{prop:intro chern classes reflexive case}
Let $\cE$ be a normalised, rank $2$, reflexive, non-locally free, toric sheaf on $\P^n$, $n\geq 3$, defined by $(c_\rho,L_\rho)_{\rho\in\Sigma(1)}$. Then the Chern classes of $\cE$ are given by
\begin{equation}
 \label{eq:intro chern formulae}
\forall k\in\N,\:  0\leq k\leq n,\:\uc_k=\sum_{\sigma\in\Sigma(k)} \prod_{\rho\in\sigma(1)}c_\rho.
\end{equation}
\end{proposition}
Here, $\sigma(1)$ is the set of rays of the $k$-dimensional cone $\sigma\in\Sigma(k)$. These expressions are  the elementary symmetric polynomials in the $(c_\rho)_{\rho\in\Sigma(1)}$'s.  From those formulae, we derive a simple proof of the Bogomolov--Gieseker inequality (Corollary \ref{cor:BogGies}), and show that a rank $2$ reflexive toric sheaf is locally free if and only if its third Chern class vanishes (Corollary \ref{cor:loc free iff c3 vanishes}). As rank $2$ toric vector bundles split (for $n\geq 3$), they are rigid. Hence we deduce that for $n\geq 3$, a rank $2$ reflexive toric sheaf will never deform to a non-split locally free one - a fact already observed by Hartshorne \cite[Introduction]{Har80}.

We then turn to the more general torsion-free torus-equivariant sheaves. Let $\cE$ be such a sheaf and consider the canonical sequence induced by the evaluation map:
$$
0\to\cE\to(\cE^\vee)^\vee\to\cQ\to 0
$$
where $\cQ$ is the quotient sheaf $(\cE^\vee)^\vee/\cE$. In general, it is not so clear how to compute the Chern classes of $\cE$ from those of $(\cE^\vee)^\vee$. To help us do so, we consider a factorization of the canonical injection by simpler injections. We will say that an injection $\cE\to\cF$ between two torsion-free sheaves of the same rank is {\it elementary} if the induced quotient $\cQ=\cF/\cE$ is pure with reduced irreducible support $S$, and restricts to a rank one sheaf on $S$. Our terminology recalls Maruyama's notion of elementary transformations for vector bundles \cite{Maru}, where the support $S$ is a divisor (but $\cQ_{\vert S}$ may not be of rank $1$). In Section \ref{sec:elementary injections}, based on Perling's description of equivariant torsion-free sheaves \cite{Per04}, and Kool's description of pure equivariant sheaves \cite{Koo11}, we obtain the following (Theorem \ref{theo:existence decomposition elementary injections} and Corollary \ref{cor:splitting geometric}) :
\begin{theorem}
 \label{theo:intro factorization}
 Let $X$ be a smooth toric variety, and let $\cE\to\cF$ be an equivariant injection between two torus-equivariant torsion-free sheaves of the same rank. Then $\cE\to \cF$ factorizes through a finite number of elementary injections, whose quotients' supports have increasing dimension.
\end{theorem}
This theorem provides a (non unique) refinement of the 
 torsion filtration \cite[Chapter 1, section 1]{HuLe}, and we expect that it should be true in a non toric setting as well. Our proof, however, is quite constructive, which helps in understanding sequences of elementary injections, which is a key aspect of what follows.

For a specific class of elementary injections on the projective space, we then obtain  in Section \ref{sec:chern classes elementary} a simple formula for the ratio of the Chern polynomials. More precisely, we will say that an elementary injection $\cE\to\cF$ is {\it saturated} if the restriction of the associated quotient $\cF/\cE$ to its support is locally free. From the toric point of view, an equivariant elementary injection comes with two parameters $(\sigma_0,m_0)\in\Sigma\times M$, where $\Sigma$ is the fan of $\C\P^n$ and $M$ its character lattice (see Definition \ref{def:elementary injection}). We then obtain the following  :
\begin{proposition}[Corollary \ref{cor:simplest product formula Chern ratios saturated elementary}]
 \label{prop:intro simplest product formula Chern ratios}
 Let  $\cE\to \cF$ be an equivariant saturated elementary  injection on the projective space with parameters $(\sigma_0,m_0)$, with $\dim(\sigma_0)\geq 1$. Then the following holds in $\Z[H]/\langle H^{n+1} \rangle$ :
  \begin{equation}
   \label{eq:intro simplest product formula Chern ratios}
  \frac{\uc(\cF)}{\uc(\cE)}=\prod_{i=0}^{\dim(\sigma_0)}(1-(\um_\Sigma+i)H)^{(-1)^i\binom{\dim(\sigma_0)}{i}},
  \end{equation}
  where $\um_\Sigma\in\Z$ is defined by Equation \eqref{eq:umbigsigma}. 
\end{proposition}
 One can show (cf Remark \ref{rem:interpretation of weights}) that in the above setting,
 $$
 \cF/\cE\simeq \iota_*\Big(\cO_{V}(-\um_\Sigma)\Big),
 $$
 where $\iota:V\to\C\P^n$ stands for the inclusion of the support $V$ of $\cF/\cE$. Note however that this isomorphism forgets about the equivariant structure of the quotient. In general, the weight $\um_\Sigma$ in Proposition \ref{prop:intro simplest product formula Chern ratios} is not entirely characterised by $m_0$, but in all our applications, $\um_\Sigma$ will simply be given by
$$
\um_\Sigma=\sum_{\rho\in\sigma_0(1)}\langle m_0, u_\rho \rangle,
$$
where $u_\rho$ is a primitive generator for $\rho$ and $\langle\cdot,\cdot\rangle$ the duality pairing. We note that while Proposition \ref{prop:intro simplest product formula Chern ratios} may be obtained directly by applying Grothendieck--Hirzebruch--Riemann--Roch theorem to $\cO_{V}(-\um_\Sigma)$, our proof relies on toric and combinatorial methods and has the advantage of giving along the way similar formulae for not necessarily saturated elementary injections (Proposition \ref{prop:ratio elementary injection}). Comparing with the general formula for the Chern polynomial of an equivariant sheaf (see Equation \eqref{eq:total chern class general}),  Formula \eqref{eq:intro simplest product formula Chern ratios} is quite simple, as it involves only a finite number of terms. In particular, using a formal logarithm, we can develop explicitly $\log(\frac{\uc(\cF)}{\uc(\cE)})$ in terms of powers of $H$ :
\begin{proposition}[Proposition \ref{prop:log expansion chern saturated elementary}]
 \label{prop:intro log expansion chern saturated elementary}
  Let $\cE\to \cF$ be an equivariant saturated elementary injection with parameters $(\sigma_0,m_0)$, with $k_0:=\dim(\sigma_0)\geq 1$. Then we have an expansion in $\Z[H]/\langle H^{n+1} \rangle$ :
 \begin{equation}
  \label{eq:intro log expansion saturated elementary}
  \log\left(\frac{\uc(\cF)}{\uc(\cE)}\right)=-\sum_{k=k_0}^n \Big(\sum_{l=k_0}^k \binom{k}{l}\, A_{l,k_0}\, \um_\Sigma^{k-l} \Big)\, \frac{H^k}{k},
 \end{equation}
where $A_{l,k}=(-1)^k\,k!\,S(l,k)$ and $S(l,k)$ is a Stirling number of the second kind.
 \end{proposition}
  
 We then consider in Section \ref{sec:applications} the prescription problem for Chern classes. Our first application provides simple obstructions for producing rank $2$ torus-equivariant sheaves with vanishing higher Chern classes. Recall that if $\cQ$ is a coherent sheaf, $\dim(\cQ):=\dim(\Supp(\cQ))$, and $\cQ$ is {\it pure} if for any non-trivial coherent subsheaf $\cF\subset\cQ$, $\dim(\cF)=\dim(\cQ)$. We denote by 
$$
0=T_{-1}(\cQ)\subset T_0(\cQ)\subset\ldots\subset T_{d-1}(\cQ)\subset T_d(\cQ)=\cQ
$$
the torsion filtration of $\cQ$, so that $d=\dim(\cQ)$ and for each $0\leq i \leq d$, the quotient $$Q_i(\cQ):=T_i(\cQ)/T_{i-1}(\cQ)$$ is either $0$ or pure of dimension $i$ (see \cite[Chapter 1]{HuLe}). Note that if $\cE$ is a torsion-free sheaf, its singular set is of codimension greater or equal to $2$ (see \cite[Chapter 2.1]{OSS}). As $\cE$ agrees with its reflexive hull away from its singular set, one has $\codim((\cE^\vee)^\vee/\cE)\geq 2$. Thanks to Theorem \ref{theo:intro factorization} and Proposition \ref{prop:intro log expansion chern saturated elementary}, we obtain in Section \ref{sec:negative results} the following :

\begin{theorem}
 \label{theo:intro obstructions smoothability}
 Let $\cE$ be a  torus-equivariant rank $2$ torsion-free sheaf with normalised reflexive hull on the projective space. Let $\cQ=(\cE^\vee)^\vee/\cE$ and set $q=\codim(\cQ)=n-\dim(\cQ)$. Then the following holds :
 \begin{enumerate}
  \item If $q\geq 4$, then at least one of $\uc_3(\cE)$ or $\uc_q(\cE)$ is non-zero.
  \item If $n\geq 3$, $q=2$, $Q_{n-3}(\cQ)=0$ and $\cE$ is semistable, then $\uc_3(\cE)\neq 0$.
 \end{enumerate}
\end{theorem}
Of course, for item $(1)$ (resp. $(2)$) in Theorem \ref{theo:intro obstructions smoothability} to hold, we assume $n\geq 4$ (resp. $n\geq 3$). Before stating the next corollary, recall that $\cE$ is said to be smoothable if there is a coherent sheaf $\tilde \cE$ on $\C\P^n\times \C$, flat over $\C$, such that $\tilde\cE_{\vert \C\P^n\times\lbrace 0\rbrace}=\cE$ and $\tilde\cE_{\vert \C\P^n\times\lbrace t\rbrace}$ is locally free for $t\neq 0$. As Chern classes are invariant under such flat deformations, we obtain the following (note that we don't need to assume $\cE$ to be normalised here) :
\begin{corollary}
 \label{cor:intro obstructions smoothability}
 Let $\cE$ be a torus-equivariant rank $2$ torsion-free sheaf on the projective space. Let $\cQ=(\cE^\vee)^\vee/\cE$ and set $q=n-\dim(\cQ)$. Assume one of the following :
 \begin{enumerate}
  \item[(i)] $n\geq 4$ and $q\geq 4$, or
  \item[(ii)] $n\geq 3$, $q=2$, $Q_{n-3}(\cQ)=0$ and $\cE$ is semistable.
 \end{enumerate}
 Then $\cE$ is not smoothable.
\end{corollary}
Again, from Corollary \ref{cor:existence reduction projective space}, the existence of indecomposable semistable rank $2$ bundles is subject to the existence of smoothable rank $2$ torsion-free equivariant sheaves with appropriate Chern classes. Then, Corollary \ref{cor:intro obstructions smoothability} points in the direction of Hartshorne's conjecture. It provides strong constraints on equivariant sheaves that could lead to potential counterexamples, showing that they tend to be very different from their simpler reflexive hull, with $Q_{n-3}(\cQ)$ non zero. Note also that Corollary \ref{cor:intro obstructions smoothability} does not hold if one drops the equivariance hypothesis, even on $\P^3$, as shown by the example $E_0$ in \cite[Proof of Proposition 7.2]{JaMaTi}\footnote{This counter-example was communicated to us by Marcos Jardim.}. This highlights the advantage of our toric approach.
 
 Despite these obstructions, we manage to provide examples with vanishing higher Chern classes. Starting from a reflexive equivariant sheaf $\cF$, whose Chern polynomial is well understood from Equation \eqref{eq:intro chern formulae}, we  systematically build sequences of saturated elementary injections (cf Proposition \ref{prop:existence sequence saturated injections}) :
 $$
 \cE=\cE_p\to\cE_{p-1}\to \ldots \to \cE_0=\cF.
 $$
Using multiplicativity of the Chern polynomial, together with Formula \eqref{eq:intro log expansion saturated elementary} for  $\log(\frac{\uc(\cE_i)}{\uc(\cE_{i+1})})$, we can try to prescribe the Chern polynomial of $\cE$, by appropriately tuning the parameters. Our main interest lies in Chern polynomials of hypothetical rank $2$ indecomposable vector bundles. Such bundles have vanishing Chern classes $\uc_i$ for $i\geq 3$, and satisfy Schwarzenberger's topological constraints \cite[Appendix 1]{Hirz}. In Section \ref{sec:prescription}, we derive the precise equations that one needs to solve in order to produce torsion-free equivariant rank $2$ sheaves whose Chern classes satisfy those constraints. As an application, we produce in Section \ref{sec:P4} two infinite families of {\it explicit} examples on $\C\P^4$, according to parity of the first Chern class.  Moreover, we explain that for infinitely many pairs $(\uc_1,\uc_2)$ that we obtain, one may not pass from one to another by twisting by a line bundle nor by pulling back along a finite cover of the projective space. More precisely :
\begin{theorem}
 \label{theo:intro P4 examples}
 For any $t\in\N^*$, there exist torus-equivariant stable rank $2$ torsion-free sheaves $\cE_t$ and $\cF_t$ on $\C\P^4$ with Chern polynomials in $H^{2\bullet}(\P^4,\Z)\simeq \Z[H]/\langle H^{5}\rangle$:
 $$\uc(\cE_t)=1+(12t+1)H+12t(3t+1)H^2$$
 and 
 $$\uc(\cF_t)=1+(12t+10)H+12(t+1)(4t+3)H^2.
 $$
\end{theorem}
Again, the Chern classes reached in Theorem \ref{theo:intro P4 examples} satisfy all known constraints imposed by locally-freeness. For $t$ large enough, they also satisfy the necessary conditions for indecomposability derived in \cite{HoSch}. By openness of semistability, any small deformation of one of those examples will remain semistable. If we luckily reach a locally free sheaf along such a deformation, we can conclude that it must be indecomposable. Indeed, if a small semistable deformation splits as $\cO_{\P^4}(k)\oplus\cO_{\P^4}(l)$, we must have $k=l$ by semistability. But then the discriminant $\Delta=4\uc_2-\uc_1^2$ vanishes, which is not the case for our examples. As another application, in Section \ref{sec:P5}, we prove
\begin{theorem}
 \label{theo:intro P5 examples}
 For any $t\in\N^*$, there exists a torus-equivariant stable rank $2$ torsion-free sheaf $\cE_t$  on $\C\P^5$ with Chern polynomial in $H^{2\bullet}(\P^5,\Z)\simeq \Z[H]/\langle H^{6}\rangle$:
 $$\uc(\cE_t)=1+(240t+1) H+((120)^2t^2+240t) H^2.$$
\end{theorem}
For those examples, the same conclusions as for the ones of Theorem \ref{theo:intro P4 examples} hold. The equations that one needs to solve in order to produce such explicit examples quickly become convoluted as the dimension of $\C\P^n$ increases. Nevertheless, in Section \ref{sec:Pn}, we derive the following existence result to illustrate our techniques:
\begin{theorem}
 \label{theo:intro higher dim examples}
 For each $n\geq 3$, there exists $m\in\N^*$ such that for all $t\in\N^*$, there is a stable rank $2$ torus-equivariant torsion-free sheaf $\cE_t$ on $\C\P^n$ with Chern polynomial
 $$
 \uc(\cE_t)=1+(2mt+1)\cdot H+(m^2t^2+2mt)\cdot H^2.
 $$
\end{theorem}
Once again, for $t$ large enough, these last examples  provide candidates for producing indecomposable rank $2$ bundles by deformations, although their construction is not as explicit as in Theorems \ref{theo:intro P4 examples} and \ref{theo:intro P5 examples}. 
\begin{remark}
 \label{rem:intro more prescription pb}
We note that while we focused on rank $2$ sheaves with Chern classes of potential  locally-free $2$-bundles in our applications, the technology developed in this paper can be used to produce stable torsion-free sheaves with more general Chern polynomials, in any rank. This is explained briefly in Section \ref{sec:more prescriptions}.
\end{remark}

\subsection*{Organisation} Section \ref{sec:background} provides background on toric sheaves and slope stability. In Section \ref{sec:chernclasses}, we compute Chern classes of rank $2$ equivariant reflexive sheaves, with some applications.  In Section \ref{sec:torsion free case}, we recall Perling's description of the category of equivariant torsion-free sheaves, introduce the notion of elementary injection in this language, and prove Theorem \ref{theo:intro factorization}. Then, the formulae for the Chern classes through elementary injections are obtained in Section \ref{sec:chern classes elementary}. Finally, the various applications are given in Section \ref{sec:applications}.

\subsection*{Acknowledgments}  
The author would like to thank Nathan Ilten, Martijn Kool and Marcos Jardim for kindly answering his questions on deformations, toric sheaves, and Hartshorne--Serre correspondence, as well as Roberto Vacca for a useful comment on a previous version of this article. We also thank the anonymous referee for their advices that improved the exposition and the content of this
paper. The author is partially supported by the grants MARGE ANR-21-CE40-0011 and BRIDGES ANR-FAPESP ANR-21-CE40-0017. He also beneficiated from the support of the French government “Investissements d'Avenir” program integrated to France 2030, under the reference ANR-11-LABX-0020-01.

\subsection*{Notations}
We will follow notations from \cite{CLS,Per04} for toric varieties and toric sheaves. In particular, we denote by $T_N=N\otimes_\Z \C^*$ the torus of a toric variety $X$, $N$ the rank $n$ lattice of its one-parameter subgroups, $M=\Hom_\Z(N,\Z)$ its character lattice and $\Sigma$ its fan of strongly convex rational polyhedral cones in $N_\R=N\otimes_\Z \R$ (see \cite[Chapter 3]{CLS}). For each $\sigma\in\Sigma$, we let $\sigma^\vee$ denote the dual cone, and $\sigma^\perp$ its orthogonal - both for the duality pairing $\langle\cdot,\cdot\rangle$. 
The variety $X$ is then covered by the $T_N$-invariant affine varieties $U_\sigma=\Spec(\C[M\cap\sigma^\vee])$, for $\sigma\in \Sigma$. We denote by $\tau\leq\sigma$ the inclusion of a face $\tau$ in $\sigma$, and $\tau<\sigma$ if the inclusion is strict (recall then that we have $U_\tau\subset U_\sigma$).  For $m\in M$, we let $\chi^m\in\C[\sigma^\vee\cap M]$ be the associated regular function on $U_\sigma$.

\section{Background on toric sheaves}
\label{sec:background}
We consider now $\C\P^n$ as a projective toric variety, for the action of the $n$-dimensional torus $T= (\C^*)^n$ given, for $\lambda=(\lambda_1,\ldots,\lambda_n)\in T$, by 
$$
\lambda\cdot [z_0:\ldots:z_n]:=[z_0:\lambda_1 z_1:\ldots:\lambda_n z_n].
$$
 We denote by $N=\Z^n$ the rank $n$ lattice of one-parameter subgroups of $T$ and by $M=\Hom_\Z(N,\Z)$ its dual lattice. We set $\Sigma$ to be the smooth and complete fan of strongly convex rational polyhedral cones in $N_\R=N\otimes_\Z \R$ associated to $\C\P^n$. In more concrete terms, let $(e_1,\ldots, e_n)$ be the canonical basis for $N_\R=\R^n$, and $e_0=-e_1-e_2\ldots-e_n$. Then, for any $I\subset \lbrace 0,1,\ldots, n\rbrace$, we obtain a cone 
 $$
 \sigma_I:=\bigoplus_{i\in I} \R_+\cdot e_i \subset N_\R,
 $$
 and $\Sigma=\lbrace \sigma_I,\: I\subsetneq \lbrace 0,1,\ldots, n\rbrace \rbrace$. We will denote $\Sigma(k)\subset \Sigma$ the set of $k$-dimensional cones, for $0\leq k\leq n$. The projective space is then covered by the $T$-invariant affine toric varieties $U_\sigma=\Spec(\C[M\cap\sigma^\vee])$, for $\sigma\in \Sigma$, where $\sigma^\vee$ denotes the dual cone of $\sigma$, that is
 $$
 \sigma^\vee=\lbrace m\in M_\R\: \vert\: \forall u\in\sigma,\;\langle m,u \rangle\geq 0 \rbrace \subset M_\R.
 $$
 
 \subsection{Reflexive equivariant sheaves}
 \label{sec:reflexive equiv sheaves}
 By definition, a coherent sheaf $\cE$ on $\C\P^n$ is $T$-equivariant if there is an isomorphism
$$\varphi : \theta^*\cE \to \pi_2^*\cE$$
satisfying some cocycle condition, where $\theta : T \times \C\P^n \to \C\P^n$ and $\pi_2 : T\times \C\P^n \to \C\P^n$ stand for the torus action and the projection on $X$ respectively (see e.g. \cite{Per04}). We will use the following definition:
\begin{definition}
 \label{def:toricsheaves}
 A {\it toric sheaf} is a $T$-equivariant coherent sheaf over $\C\P^n$.
\end{definition}
From Klyachko and Perling's work \cite{Kly90,Per04}, there is a fairly simple description of such sheaves (their work actually extend to normal toric varieties). In the reflexive case, they are uniquely determined by {\it families of filtrations}. If $\cE$ is a reflexive toric sheaf,  let's denote $(E,E^\rho(i))_{\rho\in\Sigma(1),i\in\Z}$ the associated family of filtrations, which consist of a finite dimensional complex vector space $E$ of dimension $\rank(\cE)$, together with, for each ray $\rho\in\Sigma(1)$, a bounded {\it increasing} filtration $(E^\rho(i))_{i\in\Z}$ of $E$.
The space $E$ is the fiber of $\cE$ over the identity element  $1\in T\subset \C\P^n$, while the filtrations are obtained as follows. Being reflexive,  the  sections of $\cE$ extend over codimension $2$ subvarieties. If we then set
$$
X_0=\bigcup_{\sigma\in\Sigma(0)\cup\Sigma(1)} U_\sigma,
$$ 
by the orbit-cone correspondence (\cite[Section 3.2]{CLS}), $X_0$ is the complement of $T$-orbits of co-dimension greater or equal to $2$, and thus $\cE=\iota_*(\cE_{\vert X_0})$, where $\iota : X_0 \to \C\P^n$ is the inclusion. Equivariance implies that $\cE_{\vert X_0}$ is entirely characterised by the sections $\Gamma(U_\sigma,\cE)$, for $\sigma\in\Sigma(0)\cup\Sigma(1)$. If $\sigma=\lbrace 0 \rbrace$, $U_{\lbrace 0\rbrace}=T$, and
$$
\Gamma(U_{\lbrace 0\rbrace},\cE)=E\otimes_\C \C[M].
$$
 Then, if $\rho\in \Sigma$ is a ray (i.e. a one-dimensional cone), $\Gamma(U_{\rho}, \cE)$ is graded by $$M/( M\cap \rho^\perp)\simeq\Z,$$
 where for any cone $\sigma$, 
 $$
 \sigma^\perp:=\lbrace m\in M_\R\:\vert\:\forall u\in \sigma\; \langle m,u\rangle =0\rbrace.
 $$
 As $T$ is a dense open subset of $U_\rho$, the restriction map $\Gamma(U_\rho,\cE)\to\Gamma(U_{\lbrace 0\rbrace},\cE)$ is injective and induces a $\Z$-filtration 
$$
  \ldots \subset E^\rho(i-1)\subset E^\rho(i) \subset \ldots \subset E
$$
such that one has
 \begin{equation*}
  \label{eq:sheaf from filtrations for rays}
  \Gamma(U_{\rho}, \cE)=\bigoplus_{m\in M} E^\rho(\langle m,u_\rho\rangle)\otimes \chi^m,
 \end{equation*}
 where we denote by $u_\rho$ the primitive generator of $\rho$.

\begin{remark}
Note that we will use increasing filtrations as in \cite{Per04}, rather than decreasing ones as in \cite{Kly90}.
\end{remark}
 
In the other direction, one recovers $\cE$ by setting for each $\sigma\in\Sigma$ :
\begin{equation}
  \label{eq:sheaf from family of filtrations}
  \Gamma(U_{\sigma}, \cE):=\bigoplus_{m\in M} \bigcap_{\rho\in\sigma(1)} E^\rho(\langle m,u_\rho\rangle)\otimes \chi^m.
 \end{equation}
Finally, from \cite[Theorem 2.2.1]{Kly90} (or \cite[Section 5]{Per04}),  $\cE$ will be locally free  if and only if the family of filtrations $(E^\rho(\bullet))_{\rho\in\Sigma(1)}$ satisfies Klyachko's compatibility criterion, namely that for each $\sigma\in\Sigma$, there exists a decomposition 
$$
E= \bigoplus_{[m]\in M/(M\cap\sigma^\perp)} E^\sigma_{[m]}
$$
such that for each ray  
$\rho\in\sigma(1)$ in $\sigma$ :
\begin{equation}
 \label{eq:klyachko compatibility}
E^\rho(i)=\bigoplus_{\langle m,u_\rho\rangle\leq i}  E^\sigma_{[m]}.
\end{equation}

\begin{example}
\label{ex:rank two case}
As a leading example, we describe here the family of filtrations of rank $2$ equivariant locally free sheaves on $\C\P^n$, assuming $n\geq 3$. In that case, from Kaneyama \cite[Corollary 3.5]{Kan}, the associated vector bundle splits as a direct sum of line bundles. We then start with the description of the family of filtration associated to a line bundle. By the orbit-cone correspondence \cite[Chapter 3]{CLS}, $T$-invariant  reduced and irreducible divisors on $X$ are in bijection with rays of $\Sigma$. Hence, any $T$-invariant Cartier divisor $D\subset \C\P^n$ can be written
 $$
 D=\sum_{\rho\in\Sigma(1)} d_\rho D_\rho
 $$
 where $D_\rho\subset \C\P^n$ is the orbit closure associated to $\rho\in\Sigma(1)$ and $d\rho\in\Z$.   The family of filtration associated to $\cO(D)$ for $D=\sum_{\rho\in\Sigma(1)} d_\rho D_\rho$ is given by $E=\C$ and for each $\rho\in\Sigma(1)$, and each $i\in \Z$,
  \begin{equation*}
E^\rho(i)=
\left\{ 
\begin{array}{ccc}
0 & \mathrm{ if } & i < -d_\rho,\\
\C & \mathrm{ if } & i \geq -d_\rho.
\end{array}
\right.
\end{equation*}
Consider now $\cE=\cO(D_1)\oplus\cO(D_2)$ an equivariant locally free sheaf of rank $2$, with 
$$
D_i:=\sum_{\rho\in\Sigma(1)} d^i_\rho D_\rho,
$$
for $i\in\lbrace 1,2\rbrace$. Then, by definition of direct sums, we obtain that $\cE$ comes from the family of filtrations $(\C^2,(E^\rho(\bullet))_{\rho\in\Sigma(1)})$ where this time, for each $\rho\in\Sigma(1)$ with $-d^1_\rho\leq -d^2_\rho$,
\begin{equation*}
E^\rho(i)=\left\{ 
\begin{array}{ccc}
0 & \mathrm{ if } & i < -d^1_\rho,\\
\C\times\lbrace 0\rbrace & \mathrm{ if } & -d^1_\rho\leq i < -d^2_\rho,\\
\C^2 & \mathrm{ if } &  -d_\rho^2\leq i,
\end{array}
\right.
\end{equation*}
while if $-d^2_\rho\leq -d^1_\rho$, 
\begin{equation*}
E^\rho(i)=\left\{ 
\begin{array}{ccc}
0 & \mathrm{ if } & i < -d^2_\rho,\\
\lbrace 0\rbrace \times\C & \mathrm{ if } & -d^2_\rho\leq i < -d^1_\rho,\\
\C^2 & \mathrm{ if } &  -d_\rho^1\leq i.
\end{array}
\right.
\end{equation*}
\end{example}
\subsection{Slope stability}
\label{sec:slope stability}
Recall that a torsion-free sheaf $\cE$ on $\C\P^n$ is said to be  slope stable (respectively slope semistable) if for any coherent and saturated subsheaf $\cF\subset \cE$ with $\rank(\cF)<\rank(\cE)$, one has
$$
\mu(\cF)<\mu(\cE)
$$
(respectively $\mu(\cF)\leq \mu(\cE)$),
where for any coherent torsion-free sheaf $\cG$, the slope $\mu(\cG)$ is given by
$$
\mu(\cG)=\frac{\uc_1(\cG)\cdot H^{n-1}}{\rank(\cG)}\in\Q.
$$
If $\cE$ is a reflexive toric sheaf with associated family of filtrations $(E,E^\rho(\bullet))_{\rho\in\Sigma(1)}$, from Klyachko's formula for the first Chern class (see e.g. \cite[Corollary 2.18]{CT22}) we obtain
\begin{equation}
 \label{eq:slope}
 \mu(\cE)=-\frac{1}{\rank(\cE)}\sum_{\rho\in\Sigma(1)} \iota_\rho(\cE),
\end{equation}
  where
 $$
 \iota_\rho(\cE):=\sum_{i\in\Z} i \left(\dim(E^\rho(i))-\dim (E^\rho(i-1))\right).
 $$
From Kool's work \cite[Proposition 4.13]{Koo11} (see also \cite[Proposition 2.3]{HNS19}), to check stability for $\cE$, it is enough to compare slopes with equivariant and saturated reflexive subsheaves. By \cite[Lemma 2.14]{NaTip}, any such subsheaf is associated to a family of filtrations of the form $(F, F\cap E^\rho(i))_{\rho\in\Sigma(1),i\in\Z}$ for some vector subspace $F\subsetneq E$ (see also \cite[Proposition 0.0.1]{DDK20} and \cite[Remark 1.2.4]{HNS19}). Hence we have :
\begin{proposition}
\label{prop:stability equiv}
The reflexive toric sheaf associated to $(E,E^\rho(\bullet))_{\rho\in\Sigma(1)}$ is slope semistable if and only if for any vector subspace $F\subsetneq E$, we have
$$
\frac{1}{\dim(F)}\sum_{\rho\in\Sigma(1)} \iota_\rho(F)\geq \frac{1}{\dim(E)}\sum_{\rho\in\Sigma(1)} \iota_\rho(\cE)
$$
and slope stable if the above inequalities are all strict.
 \end{proposition}
 In the previous statement, we used the notation 
$$
 \iota_\rho(F):=\displaystyle\sum_{i\in\Z} i \left(\dim(F\cap E^\rho(i))-\dim (F\cap E^\rho(i-1))\right).
$$
Let's assume now that $\rank(\cE)=2$. Up to an isomorphism, we can assume $E=\C^2$. For each $\rho\in\Sigma(1)$, there exist integers $a_\rho \leq b_\rho$ and a line $L_\rho\subset\C^2$ such that 
 $$
 E^\rho(i)=\left\{
 \begin{array}{ccc}
                  \lbrace 0 \rbrace & \mathrm{ if } & i< a_\rho \\
                 L_\rho & \mathrm{ if } & a_\rho\leq i< b_\rho \\
\C^2 & \mathrm{ if } & b_\rho\leq i .
                  \end{array}
                  \right.
 $$
 Then, by Proposition \ref{prop:stability equiv}, we deduce:
 \begin{corollary}
  \label{cor:slope stable rank two}
  With the previous notations, $\cE$ is slope semistable (respectively slope stable) if and only if for any one-dimensional subspace $L\subset \C^2$, 
  $$
  \sum_{L_\rho = L} c_\rho \leq \sum_{L_\rho \neq L} c_\rho,
  $$
  (respectively $\sum_{L_\rho = L} c_\rho < \sum_{L_\rho \neq L} c_\rho$), where $c_\rho:=b_\rho-a_\rho$.
 \end{corollary}

\subsection{Equivariant sheaves as fixed points in Gieseker moduli spaces}
\label{sec:new approach}
In this section, we reduce Hartshorne's conjecture to the study of smoothability of torsion-free torus equivariant sheaves, in the semistable case.

\begin{proposition}
 \label{prop:reduction to smoothable toric sheaves general}
 Let $(X,L)$ be a smooth polarised toric variety. Then there exists a Gieseker semistable vector bundle $E\to X$ with Hilbert polynomial $P$ if and only if there exists a smoothable, Gieseker semistable, torus equivariant sheaf with Hilbert polynomial $P$.
\end{proposition}
We refer to \cite[Chapters 1, 2 and 4]{HuLe} for the definition of Gieseker stability and for the construction of the moduli space of Gieseker semistable sheaves with given Hilbert polynomial.
\begin{proof}
The {\it if} direction of the proof is clear, given openness of semistability \cite[Proposition 2.3.1]{HuLe}. To prove the converse, we will use Kool's work \cite{Koo11}. We follow here the notations from \cite[Section 4]{Koo11}. In particular, we denote by $M^{ss}$ the moduli space of Gieseker semistable torsion-free sheaves with Hilbert polynomial $P$. This can be obtained as a good categorical quotient $R\to M^{ss}$ of an open subscheme $R\subset Q$ of a Quot scheme $Q$ by a reductive group $G$. Given a linearisation $T\curvearrowright L$ of the action of the torus $T$ of $X$, one can construct a regular $T$-action on $Q$ that restricts to $R$, commutes with the $G$-action, and descends to a regular action on $M^{ss}$ (cf \cite[Proposition 4.1]{Koo11}). On a closed point $[\cE]\in M^{ss}_{cl}$, this action simply reads $[ \cE]\mapsto [t^*\cE]$ for $t\in T$.

 Assume that $E$ is a Gieseker semistable vector bundle with Hilbert polynomial $P$. Then, $E$ yields a point $[E]\in M^{ss}$. If $[E]$ is a torus fixed point, then by definition of the $T$-action, and as closed points in $M^{ss}$ correspond to $S$-equivalence classes, we have for any $t\in T$, $\Gr(t^*E)\simeq \Gr(E)$ (where $\Gr$ stands for the graded object). We then obtain $t^*\Gr(E)\simeq\Gr(t^*E)\simeq \Gr(E)$. Decompose $\Gr(E)=\oplus_i E_i^{r_i}$ into irreducible stable components. We argue now as in \cite[Proposition 3.19]{Koo11}. By the Krull-Schmidt property of the category of coherent sheaves on the projective variety $X$, we deduce that for each $t\in T$, there exists $j$ such that $t^* E_i\simeq E_j$. Set $\Gamma_j=\lbrace t\in T\:\vert\: t^*E_i\simeq E_j\rbrace$. From stability, and as the stable components of $\Gr(E)$ share the same Hilbert polynomial, $\Gamma_j=\lbrace t\in T\:\vert\: \dim(\Hom(t^*E_i,E_j))\geq 1 \rbrace$. By semi-continuity, we deduce that $\Gamma_j$ is closed in $T$, hence open as well as $T\setminus \Gamma_j=\bigcup_{k\neq j} \Gamma_k$. As $\Gamma_i$ contains the identity, we deduce that for all $t\in T$, $t^*E_i \simeq E_i$. Then, by \cite[Propositions 4.2 and 4.4]{Koo11}, each $E_i$ admits an equivariant structure, and so does $\Gr(E)$. By \cite[theorem 4.3.3]{HuLe}, up to an isomorphism, $\Gr(E)$ lies in the $G$-orbit closure of (the point corresponding to) $E$ in $Q$, and its $G$-orbit is closed. Then, by a result of Kempf \cite[Theorem 1.4]{Kempf78}, we may find a one parameter subgroup $\lambda :\C^*\to G$ such that
  $$
 \lim_{t\to 0} \lambda(t)\cdot E\simeq \Gr(E),
 $$
 where we identify the sheaves and their corresponding points in the Quot scheme. This provides a curve $c:\C\to Q$ with $c(0)=\Gr(E)$ and $c(t)$ isomorphic to $\lambda(t)\cdot E$ for $t\neq 0$. Pulling back the universal family to $\C$ yields a smoothing of the equivariant sheaf $\Gr(E)$, and concludes the proof when $[E]$ is a $T$-fixed point.

 Assume now that $[E]$ is not a torus fixed point. Then we can consider the $T$-orbit closure of $[E]$ in the projective scheme $M^{ss}$. By Borel's fixed point theorem \cite[Proposition 15.5]{Borel}, one may find a torus-fixed point, say $[\cF]$, in this closure.  Then, consider the $T\times G$-action\footnote{Note that the $T$-action is a left action, while the $G$-action is a right action, but we can precompose the $T$-action by the inverse map to obtain a genuine right action of the product} on $R\subset Q$.  Denote by $\pi:R\to M^{ss}$ the GIT quotient map. As $\pi:R\to M^{ss}$ is a good quotient, it sends $G$-invariant closed sets to closed sets. It also commutes with the $T$-action. By abuse of notation, if $E\in R$ denotes the point in $Q$ corresponding to $E$, the $T\times G$-orbit closure   $\overline{(T\times G)\cdot E}^R$ in $R$ is closed and $G$-invariant, and contains $E$.  By the properties of $\pi$, we hence deduce that in $M^{ss}$:
 \[
 \overline{T\cdot [E]}\subset \pi\big(\overline{(T\times G)\cdot E}^{R}\big).
 \]
 On the other hand, by continuity, we have
\[
 \pi\big(\overline{(T\times G)\cdot E}^{R}\big)\subset \overline{T\cdot [E]}
 \] 
 (recall that $\pi(G\cdot E)=[E]$ and the $T$ and $G$ actions commute).  Hence \[
 [\cF]\in\overline{T\cdot [E]} =\pi\big(\overline{(T\times G)\cdot E}^{R}\big).
 \]
 Then, there exists $\cF'\in \overline{(T\times G)\cdot E}^{R}$ such that $\pi(\cF')=[\cF]$. By construction, $\cF'$ is $S$-equivalent to $\cF$. As 
 \[
  \overline{G\cdot\cF'}\subset \overline{(T\times G)\cdot E}^{R}
 \]
 by closedness and $G$-invariance, we deduce using \cite[theorem 4.3.3]{HuLe} that $\Gr(\cF)$ lives in the $T\times G$-orbit closure of $E$. Using Kempf's \cite[Theorem 1.4]{Kempf78} again, we obtain a smoothing of $\Gr(\cF)$. Using the same argument as for $\Gr(E)$ in the previous case, $\Gr(\cF)$ carries a $T$-equivariant structure, which concludes the proof.
\end{proof}
On the projective space, the Chern polynomial determines, and is fixed by, the Hilbert polynomial. Hence we deduce the following.
\begin{corollary}
 \label{cor:existence reduction projective space}
 There exists a Gieseker semistable vector bundle on $\P^n$ with Chern polynomial $\uc$ if and only if there exists a Gieseker semistable, smoothable, torus-equivariant, torsion-free sheaf with Chern polynomial $\uc$.
\end{corollary}
By the Bogomolov--Gieseker inequality, if $\cE$ is semistable, one has $\Delta(\cE)=4\uc_2-\uc_1^2\geq 0$. On the other hand, if a rank $2$ vector bundle splits, then $\Delta(\cE)\leq 0$. Hence, a semistable rank $2$ vector bundle with strictly positive $\Delta$ is necessarily indecomposable. In the rank $2$ semistable case, we further note that Hartshorne's conjecture  reduces to the case of zero first Chern class.
\begin{proposition}
 \label{prop:reduction even case}
 Assume that there is no indecomposable, rank $2$, Gieseker semistable vector bundle with zero first Chern class on $\P^n$. Then there is no indecomposable, rank $2$, Gieseker semistable vector bundle on $\P^n$.
\end{proposition}

\begin{proof}
As semistability and indecomposability are preserved upon tensorisation by a line bundle, the assumption directly implies that any Gieseker semistable indecomposable rank $2$ bundle must have an odd first Chern class. Assume then by contradiction that $E$ is an indecomposable, rank $2$, Gieseker semistable vector bundle on $\P^n$ with odd first Chern class. Consider the degree $2^n$ ramified cover $f :\P^n \to \P^n$ given in coordinates by $f([z_0:\ldots :z_n])=[z_0^2:\ldots:z_n^2]$. The vector bundle $f^* E$ has even first Chern class $2\uc_1(E)$. As $E$ is semistable, $\Delta(E)\geq 0$. As $\uc_1(E)$ is odd, $\Delta(E)$ is odd, hence $\Delta(E) >0$. Then, $\Delta(f^*E)=4\Delta(E) >0$. We deduce that $f^*E$ is not split. It remains to show that $f^*E$ is Gieseker semistable to conclude. As $\uc_1$ is odd, by \cite[Lemma 1.2.13 and 1.2.14]{HuLe}, $E$ is in fact slope stable. Hence, by \cite[Lemma 3.2.3]{HuLe}, $f^*E$ is slope polystable. As it is indecomposable, it is actually slope stable. The result follows.
\end{proof}
Hence, by Corollary \ref{cor:existence reduction projective space} and Proposition \ref{prop:reduction even case}, in the semistable case, Hartshorne's conjecture for indecomposable rank $2$ vector bundles reduces to showing that there is no smoothable, semistable, torus-equivariant, rank $2$ torsion-free sheaf with vanishing $\uc_1$.

\section{Chern classes in the reflexive case}
\label{sec:chernclasses}
In this section, we compute Chern classes for rank $2$ reflexive toric  sheaves over $\P^n$, with $n\geq 3$. Let $\cE$ be such a sheaf, with family of filtration $(\C^2,E^\rho(\bullet))_{\rho\in\Sigma(1)}$ given as in Section \ref{sec:slope stability} by
 \begin{equation}
  \label{eq:family filtrations E}
 E^\rho(i)=\left\{
 \begin{array}{ccc}
                  \lbrace 0 \rbrace & \mathrm{ if } & i< a_\rho \\
                 L_\rho & \mathrm{ if } & a_\rho\leq i< b_\rho \\
\C^2 & \mathrm{ if } & b_\rho\leq i .
                  \end{array}
                  \right.
 \end{equation}
As observed in Example \ref{ex:rank two case},
if $\cE$ is locally free, it is isomorphic to $\cO(D_1)\oplus\cO(D_2)$ for some $T$-invariant Cartier divisors $D_1$ and $D_2$. Then, its Chern polynomial is given by 
$$
\uc(\cE)=(1+\uc_1(\cO(D_1))(1+\uc_1(\cO(D_2))=1+\uc_1 H+\uc_2 H^2
$$
with $\uc_1=\uc_1(\cO(D_1))+\uc_1(\cO(D_2))$ and $\uc_2=\uc_1(\cO(D_1))\,\uc_1(\cO(D_2))$.
Moreover, we see that the set of lines $\lbrace L_\rho\:\vert\: \rho\in\Sigma(1)\rbrace$ used to describe the family of filtrations of $\cE$ contains at most two distinct elements (given explicitly by $\C\cdot(1,0)$ and $\C\cdot(0,1)$ in the notations of Example \ref{ex:rank two case}). In the other direction, we have
\begin{lemma}
 \label{lem:singular three xrhos}
 Assume that the set $\lbrace L_\rho\:\vert\: \rho\in\Sigma(1)\rbrace$  contains at least three distinct lines $L_{\rho_1}, L_{\rho_2}$ and $L_{\rho_3}$, with $a_{\rho_i}<b_{\rho_i}$ for $i\in\lbrace 1,2,3\rbrace$. Then $\cE$ is not locally free.
\end{lemma}
The proof of this lemma can be found in \cite[Proof of Proposition 4]{TipTrGr}, and is a direct application of Klyachko's locally freeness criterion applied to the three dimensional cone spanned by the $\rho_i$ (we work on $\P^n$, with $n\geq 3$). Until the end of this section, we will assume that $\cE$ is not locally-free, and derive a formula for its total Chern class
$$
\uc(\cE)=1+\uc_1H+\ldots+\uc_n H^n.
$$

\subsection{Perling's resolution}
\label{sec:perling resolution}
We will use a resolution of $\cE$ by direct sums of line bundles. Such a resolution has been constructed by Perling in \cite[Chapters 5 and 6]{Per03} (who considered equivariant reflexive sheaves over toric varieties in general). Consider the $(n+1)$-dimensional complex vector space
$$
F:=\bigoplus_{\rho\in\Sigma(1)}L_\rho,
$$
where in this sum we don't consider the $L_\rho$'s as subspaces of $\C^2$. For each $\rho$, we have an inclusion map $\iota_\rho : L_\rho \to \C^2$, leading to a surjective morphism 
$$
\begin{array}{cccc}
 \iota : & F & \to & \C^2 \\
          & (x_\rho)_{\rho\in\Sigma(1)} & \mapsto & \displaystyle\sum_{\rho\in\Sigma(1)} \iota_\rho(x_\rho).
\end{array}
$$
Note that surjectivity comes from our hypothesis of $\cE$ being non locally-free, cf Lemma \ref{lem:singular three xrhos}. We then consider the rank $(n+1)$ reflexive toric sheaf $\cF$ whose family of filtrations is given by $(F,F^\rho(\bullet))_{\rho\in\Sigma(1)}$, where
$$
F^\rho(i)=\left\{
 \begin{array}{ccc}
                  \lbrace 0 \rbrace & \mathrm{ if } & i< a_\rho \\
                 L_\rho & \mathrm{ if } & a_\rho\leq i< b_\rho \\
F & \mathrm{ if } & b_\rho\leq i .
                  \end{array}
                  \right.
$$
Note that the map $\iota$ sends $F^\rho(i)$ to $E^\rho(i)$ for each $(i,\rho)\in \Z\times\Sigma(1)$. From Klyachko's construction, $\iota$ induces a surjective morphism (still denoted $\iota$) from $\cF$ to $\cE$.  From Klyachko's criterion for locally freeness, $\cF$ is locally free. Indeed, for a maximal dimensional cone $\sigma\in\Sigma(n)$ with rays $\sigma(1)$, set $(u_\rho^*)_{\rho\in\sigma(1)}$ to be the basis of $M$ dual to $(u_\rho)_{\rho\in\sigma(1)}$. We then have a  splitting
\[
 F=\bigoplus_{\rho\in\Sigma(1)}L_\rho,
\]
where we assign to $L_\rho$ the weight $a_\rho u_\rho^*+\sum_{\rho'\in\sigma(1)\setminus\lbrace\rho\rbrace} b_{\rho'} u_{\rho'}^*$ if $\rho\in\sigma(1)$, and the weight $\sum_{\rho'\in\sigma(1)} b_{\rho'} u_{\rho'}^*$ if $\rho\notin \sigma(1)$. It is straightforward to check that Klyachko's compatibility condition \eqref{eq:klyachko compatibility} is satisfied. More precisely, we have
\begin{equation}
 \label{eq:sheaf F}
\cF\simeq \bigoplus_{\rho\in\Sigma(1)}\cO(-a_\rho D_\rho-\sum_{\rho'\in\Sigma(1)\setminus \lbrace \rho\rbrace}b_{\rho'}D_{\rho'}).
\end{equation}
Denote by $\cK$ the kernel of $\iota : \cF\to \cE$. It is a reflexive toric sheaf with family of filtrations $(K,K^\rho(\bullet))_{\rho\in\Sigma(1)}$ given by $K=\ker (\iota : F\to E)$ and
$$
K^\rho(i)=\left\{
 \begin{array}{ccc}
                  \lbrace 0 \rbrace & \mathrm{ if } & i< b_\rho \\
K & \mathrm{ if } & b_\rho\leq i .
                  \end{array}
                  \right.
$$
We obtain that $\cK$ is the rank $(n-1)$ sheaf obtained as $(n-1)$ copies of a line bundle :
\begin{equation}
 \label{eq:sheaf K}
\cK\simeq (\cO(-\sum_{\rho\in\Sigma(1)}b_\rho D_\rho))^{\oplus (n-1)}.
\end{equation}
To conclude, we proved 
\begin{lemma}
 \label{lem:resolution}
 We have a short exact sequence of toric sheaves
 $$
 0\to (\cO(-\sum_{\rho\in\Sigma(1)}b_\rho D_\rho))^{\oplus (n-1)}\to \bigoplus_{\rho\in\Sigma(1)}\cO(-a_\rho D_\rho-\sum_{\rho'\in\Sigma(1)\setminus \lbrace \rho\rbrace}b_{\rho'}D_{\rho'}) \to \cE \to 0.
 $$
\end{lemma}
\begin{remark}
 The result of Lemma \ref{lem:resolution} actually holds true on any smooth toric variety.
\end{remark}

\subsection{The Chern polynomial}
\label{sec:chern polynomial}
From the previous section, we can compute the total Chern class of $\cE$. We keep the notations for the families of filtrations of $\cE$, cf Equation \eqref{eq:family filtrations E}. We will set $$c_\rho:=b_\rho-a_\rho\in\N$$ for each $\rho\in\Sigma(1)$, together with
$$
a:=\sum_{\rho\in \Sigma(1)} a_\rho,\: b:=\sum_{\rho\in \Sigma(1)} b_\rho,\: \mathrm{and} \:
c:=\sum_{\rho\in \Sigma(1)} c_\rho.
$$
\begin{proposition}
 \label{prop:chern class}
 The total Chern class of $\cE$ is given by the following expression in the ring $\Z[H]/\langle H^{n+1}\rangle$ :
 \begin{equation}
  \label{eq:total Chern class}
 \displaystyle \uc(\cE)=(1-b H)^2 \prod_{\rho\in\Sigma(1)}\left(1+ \frac{c_\rho H}{1-bH}\right)
 \end{equation}
\end{proposition}
\begin{proof}
 Using the short exact sequence
 $$
 0\to\cK\to\cF\to\cE\to 0
 $$
 from Lemma \ref{lem:resolution}, we obtain
 $$
 \uc(\cE)=\frac{\uc(\cF)}{\uc(\cK)}
 $$
 (note that all total Chern classes satisfy $\uc_0=1$, hence are invertible modulo $H^{n+1}$). From the Equation \eqref{eq:sheaf F} we obtain
 $$
 \begin{array}{ccc}
 \uc(\cF) & = & \displaystyle\prod_{\rho\in\Sigma(1)}(1-(a_\rho+\sum_{\rho'\neq \rho}b_{\rho'})H)\\
          & = &\displaystyle \prod_{\rho\in\Sigma(1)}(1-(b-c_\rho)H).
 \end{array}
 $$
 On the other hand, with \eqref{eq:sheaf K}, we have
 $$
 \uc(\cK)=(1-bH)^{n-1},
 $$
 and the result follows.
\end{proof}
Using
$$
\frac{1}{1-bH}=1+bH+b^2H^2+\ldots+b^nH^n,
$$
the following corollary is then straightforward.
\begin{corollary}
 \label{cor:first chern classes}
 We have
 $$
 \uc_1=-(a+b),\: \uc_2=ab+\sum_{\underset{\tau=\rho+\rho'}{\tau\in\Sigma(2)}}c_{\rho}c_{\rho'}
 $$
 and 
 $$
 \uc_3=\sum_{\underset{\tau=\rho+\rho'+\rho''}{\tau\in\Sigma(3)}}c_{\rho}c_{\rho'}c_{\rho''}\geq 0.
 $$
\end{corollary}
The above formula for $c_3$ provides the following, that was already observed by Hartshorne in the case of (non necessarily equivariant) reflexive sheaves of rank $2$ on $\C\P^3$ (\cite[Proposition 2.6]{Har80}) :
\begin{corollary}
 \label{cor:loc free iff c3 vanishes}
 A rank $2$ reflexive toric sheaf is locally free if and only if its third Chern class vanishes.
\end{corollary}
\begin{proof}
 From Lemma \ref{lem:singular three xrhos}, if $\cE$ is a non locally free reflexive toric sheaf of rank $2$ (with the previous notations), there are at least three different rays $\lbrace \rho,\rho',\rho''\rbrace\subset\Sigma(1)$ with $c_\rho, c_{\rho'}, c_{\rho''}$ non zero. Then, from the formula in Corollary \ref{cor:first chern classes}, we see that $\uc_3 >0$. On the other hand, if a rank $2$ sheaf is locally free, it is well known that its $k$-th Chern class vanishes for $k\geq 3$.
\end{proof}
Another well-known fact that we recover easily from our computations is the Bogomolov-Gieseker inequality:
\begin{corollary}
\label{cor:BogGies}
Assume that $\cE$ is semistable. Then
$$
\Delta(\cE)=4\uc_2-\uc_1^2\geq 0.
$$
\end{corollary}
\begin{proof}
 As $\Delta(\cE)$ is invariant under tensoring $\cE$ by a line bundle, we may assume, tensoring by $\cO(\sum_{\rho\in\Sigma(1)} a_\rho D_\rho )$, that $a_\rho=0$ for each $\rho\in\Sigma(1)$. Then, for each $\rho\in\Sigma(1)$, $b_\rho=c_\rho$, and $a=0$, $\uc_1=- b=-c$, $\uc_2=\sum_{\underset{\tau=\rho+\rho'}{\tau\in\Sigma(2)}}c_{\rho}c_{\rho'}$.
 Hence,
 $$
 \Delta(\cE)=2\sum_{\underset{\tau=\rho+\rho'}{\tau\in\Sigma(2)}}c_{\rho}c_{\rho'}-\sum_{\rho\in\Sigma(1)} c_\rho^2.
 $$
 From Section \ref{sec:slope stability}, slope semistability implies that for all $\rho\in\Sigma(1)$,
 $$
 \sum_{L_{\rho'}=L_\rho} c_{\rho'}\leq \sum_{L_{\rho'}\neq L_\rho} c_{\rho'}.
 $$
 Multiplying by $c_\rho$ gives
 $$
 c_\rho^2\leq \sum_{L_{\rho'}=L_\rho} c_\rho c_{\rho'}\leq \sum_{L_{\rho'}\neq L_\rho} c_\rho c_{\rho'}.
 $$
 Summing over $\rho\in\Sigma(1)$ leads to 
 $$
 \sum_{\rho\in\Sigma(1)}  c_{\rho}^2\leq \sum_{\rho\in\Sigma(1)} \sum_{L_{\rho'}\neq L_\rho} c_\rho c_{\rho'}\leq \sum_{\rho\in\Sigma(1)} \sum_{\rho'\neq \rho} c_\rho c_{\rho'}.
$$
 As the right hand term in this last inequality is twice the sum $\sum_{\underset{\tau=\rho+\rho'}{\tau\in\Sigma(2)}}c_{\rho}c_{\rho'}$, the result follows.
\end{proof}
We now give formulae for higher Chern classes. We first introduce for any integer $k\in \N^*$,
$$
s_k:=\sum_{\sigma\in\Sigma(k)} \prod_{\rho\in\sigma(1)}c_\rho.
$$
In particular, we have $\uc_2=ab+s_2$ and $\uc_3=s_3$. More generally,
\begin{lemma}
 \label{lem:chern classes general formula}
 Let $k\geq 3$. Then
 \begin{equation}
  \label{eq:chern classes general}
 \uc_k=\sum_{i=0}^{k-3} \binom{k-3}{i} s_{i+3}\,b^{k-(i+3)}.
 \end{equation}

\end{lemma}
\begin{proof}
 We first set 
 $$
 p(H)=\prod_{\rho\in\Sigma(1)}\left(1+ \frac{c_\rho H}{1-bH}\right),
 $$
 so that 
 \begin{equation}
  \label{eq:chern and p}
 \uc(\cE)=(1-2bH+b^2H^2)p(H).
 \end{equation}
 Then we have
 $$
 \begin{array}{ccc}
 p(H) & = & \displaystyle\prod_{\Sigma(1)}(1+c_\rho H+c_\rho b H^2 +c_\rho b^2 H^3 + c_\rho b^3 H^4+ \ldots) \\
      & = &\displaystyle 1 + \sum_{j \geq 1 } \left(\sum_{i=1}^j \binom{j-1}{i-1}s_ib^{j-i} \right) H^j.
 \end{array}
$$
Introduce the notation 
$$
p(H)=1+p_1 H +p_2H^2+\ldots +p_n H^n,
$$
then from \eqref{eq:chern and p} we deduce for $k\geq 3$ :
$$
\uc_k=p_k-2b\,p_{k-1}+b^2\,p_{k-2}.
$$
Hence 
$$
\uc_k=\sum_{i=1}^k \binom{k-1}{i-1}s_i\,b^{k-i}-2b\sum_{i=1}^{k-1}\binom{k-2}{i-1}s_i\,b^{k-1-i}+b^2\sum_{i=1}^{k-2}\binom{k-3}{i-1} s_i\, b^{k-2-i}.
$$
The result is then straightforward, and follows from Pascal's rule for the binomial coefficients.
\end{proof}

\subsection{Normal forms}
\label{sec:normal forms}
One may consider, up to tensoring by a line bundle, two natural normalisations for $\cE$. The first one sets $a_\rho=0$ for all $\rho\in\Sigma(1)$. Up to tensoring $\cE$ by $\cO(\sum_{\rho}a_\rho\, D_\rho)$ (which doesn't affect slope (semi)stability)), we may assume that $a_\rho=0$ for all $\rho\in\Sigma(1)$, see \cite[Remark 2.2.15]{DDK20}. We then have 
$$
 E^\rho(i)=\left\{
 \begin{array}{ccc}
                  \lbrace 0 \rbrace & \mathrm{ if } & i< 0 \\
                  L_\rho & \mathrm{ if } & 0\leq i< b_\rho \\
\C^2 & \mathrm{ if } & b_\rho\leq i .
                  \end{array}
                  \right.
 $$
 For the second one, tensoring by $\cO(\sum_{\rho}b_\rho\, D_\rho)$ instead, we may assume that $b_\rho=0$ for all $\rho\in\Sigma(1)$. In the first case, we obtain the following inequalities.
 \begin{lemma}
  \label{lem:inequalities}
  Assume that $\cE$ is normalized so that $a_\rho=0$ for all $\rho\in\Sigma(1)$. Then for any $k\geq 3$,
  $$
  \sum_{i=0}^{k-3} \binom{k-3}{i}\uc_{i+3}\uc_1^{k-i-3}\geq 0.
  $$
 \end{lemma}
\begin{proof}
 By a similar proof as in Lemma \ref{lem:chern classes general formula}, one obtains by induction from Equation \eqref{eq:chern classes general} :
 $$
 s_k=\sum_{i=0}^{k-3} \binom{k-3}{i}\uc_{i+3}\uc_1^{k-i-3}.
 $$
 As $s_k\geq 0$ by definition, the result follows.
\end{proof}
With the second normalisation, the total Chern class becomes quite transparent from the data (the proof follows directly from Lemma \ref{lem:chern classes general formula}) :
\begin{corollary}
 \label{cor:chern classes normalized}
 Assume that $\cE$ is normalized so that $b_\rho=0$ for all $\rho\in\Sigma(1)$. Then for any $k\in \N^*$,
 $$
 \uc_k=s_k=\sum_{\sigma\in\Sigma(k)} \prod_{\rho\in\sigma(1)}c_\rho.
 $$
\end{corollary}
\begin{remark}
 \label{rem:relation to Payne's formula equiv Chern class}
 Using Payne's formula for equivariant Chern classes of toric vector bundles \cite[Proposition 3.1]{payne2008} together with Perling's resolution, one may easily derive the equivariant Chern classes of $\cE$. Specialisation then recovers our formulae. More generally, this method applies to compute equivariant Chern classes of rank two reflexive toric sheaves on smooth toric varieties.
\end{remark}

An interesting consequence of this formula is the following.
\begin{proposition}
 \label{prop:prescribing chern classes}
 Let $Q(H)\in \Z[H]/\langle H^{n+1}\rangle$. If $Q$ lifts to a polynomial $\tilde Q\in \Z[H]/\langle H^{n+2}\rangle$ that splits as 
 $$\tilde Q=(1+r_0H)(\ldots)(1+r_nH)$$
 for $(r_i)_{0\leq i\leq n}\in \N^{n+1}$, then there exists a reflexive toric sheaf $\cE$ with total Chern class $Q(H)$.
\end{proposition}
\begin{proof}
 Simply take $\cE$ to be a reflexive toric sheaf whose family of filtrations satisfies $b_\rho=0$ for all $\rho\in\Sigma(1)$ and $(c_\rho)_{\rho\in\Sigma(1)}=(r_i)_{0\leq i \leq n}$. Then from Corollary \ref{cor:chern classes normalized}, the result follows, as the $s_k$'s are the elementary symmetric polynomials in the $c_\rho$'s.
\end{proof}

We end this section with a geometric interpretation of the second normalisation.
\begin{proposition}
 \label{prop:geometric interpretation normalisation}
 Let $\rho\in\Sigma(1)$. Then $b_\rho=0$ if and only if the restriction map on weight $m$ sections $\Gamma(U_\rho,\cE)_m\to\Gamma(T,\cE)_m$ is an isomorphism provided $\langle m,u_\rho \rangle \geq 0$.
\end{proposition}
In other words, $b_\rho=0$ means that all weight $m$ sections over $T$ extend over an invariant divisor $D_\rho$ if and only if $\langle m,u_\rho \rangle \geq 0$, that is the weight of the torus action in a transverse direction to $D_\rho$ is positive. The proof of this proposition is a direct reformulation of Klyachko's description for toric sheaves. As we won't use it in what follows, we leave the proof to the reader.

\section{Elementary injections and factorization}
\label{sec:torsion free case}

From the previous sections, we see that a reflexive torus-equivariant sheaf will never deform to an indecomposable rank $2$ locally free sheaf on the projective space. To hope for such deformations, one may consider the more general class of equivariant torsion-free sheaves. If $\cE$ is such a sheaf, then one has the exact sequence :
$$
0\to \cE \to \cF\to \cQ\to 0
$$
where $\cF=(\cE^\vee)^\vee$ stands for the reflexive hull of $\cE$ and $\cQ$ for the quotient. From the previous sections, $\cF$ is fairly well understood. In the other direction, we will start from a reflexive toric sheaf $\cF$ and build torsion-free equivariant subsheaves whose reflexive hulls agree with $\cF$. This may be done by considering a sequence of {\it elementary injections} 
$$\cE_0\to\cE_1\to \ldots\to \cE_{p-1}\to\cE_p=\cF$$
where the Chern classes of $\cE_{i-1}$ are easily computed from those of $\cE_i$. In this section, we will define such elementary injections and provide some of their properties. We will then show that any injection between two torus-equivariant torsion-free sheaves of the same rank factorizes through elementary injections, on any smooth projective variety.

\subsection{Equivariant torsion-free sheaves}
\label{sec:torsion free and families filtrations}
We will recall here the combinatorial description of equivariant torsion-free sheaves over toric varieties, following Perling \cite[Section 5]{Per04}.  Such a sheaf $\cE$ is entirely described by its rings of equivariant sections over invariant affine subsets, indexed by cones $\sigma\in\Sigma$ and weights $m\in M$:
$$
E^\sigma_m:=\Gamma(U_\sigma,\cE)_m.
$$
As in the reflexive case, we let $E\simeq\C^r$ be the stalk over the point $1\in T\subset \C\P^n$, so that
$$
E^0=\C[M]\otimes E
$$
and $E^0_m\simeq \C^r$ for all weights $m\in M$. For a given cone $\sigma\in\Sigma$, and a face $\tau<\sigma$, we have injective restriction maps
\begin{equation}
 \label{eq:injective restrictions}
E^\sigma_m=\Gamma(U_\sigma,\cE)_m \hookrightarrow \Gamma(U_\tau,\cE)_m=E^\tau_m,
\end{equation}
so we can see each vector space $E^\sigma_m$ as a vector subspace of $E\simeq\C^r$. Then, for each $\sigma\in\Sigma$, we consider the relation $\leq_\sigma$ on $M$ defined by 
$$
m\leq_\sigma m' \Longleftrightarrow m'-m\in\sigma^\vee,
$$
and we denote $m <_\sigma m'$ if $m\leq_\sigma m'$ and $m'\nleq_\sigma m$ (where the latter means that $m-m'\notin\sigma^\vee$).
The coordinate function $\chi^{m'-m}\in\C[\sigma^\vee\cap M]$ induces an injection 
\begin{equation}
 \label{eq:injective weight multiplication}
\chi^{m'-m} : E^\sigma_m= \Gamma(U_\sigma,\cE)_m \to \Gamma(U_\sigma,\cE)_{m'}= E^\sigma_{m'},
\end{equation}
that we may assume to be an inclusion, up to isomorphism.  Together with coherence, and torsion-freeness, we reach the following (see \cite[Section 5.4]{Per04}) :
\begin{definition}
 \label{def:perling}
 A {\it family of multifiltrations} for a $\C$-vector space $V$ is a collection for each $\sigma\in\Sigma$ of vector subspaces $(E^\sigma_m)_{m\in M}$ of $V$ satisfying the following :
 \begin{enumerate}
  \item[$(i)_f$] For any $\sigma\in\Sigma$ and any $m\leq_\sigma m'$, $E^\sigma_m\subset E^\sigma_{m'}$,
  \item[$(ii)_f$] For any $m\in M$, $E^0_m=V$,
  \item[$(iii)_f$] For each chain $\ldots <_\sigma m_{i-1} <_\sigma m_i <_\sigma\ldots $ in $M$ there is $i_0\in\Z$ such that $E_{m_i}^\sigma=0$ for $i\leq i_0$,
  \item[$(iv)_f$] There exists only finitely many vector spaces $E^\sigma_m$ which are not contained in the union $\displaystyle\bigcup_{m'<_\sigma m}E^\sigma_{m'}$,
  \item[$(v)_f$] For each cone $\sigma\in\Sigma$ and each facet $\tau<\sigma$ there is $m_\tau\in \sigma^\vee\cap\tau^\perp$ such that $\tau^\perp\cap M=\tau^\perp\cap M+\N\cdot (-m_\tau)$. Then, for any $m\in M$, $E^\tau_m=\displaystyle \bigcup_{i\in\N} E^\sigma_{m+i m_\tau}$.
 \end{enumerate}
\end{definition}
Note that points $(ii)_f$ and $(v)_f$ in Definition \ref{def:perling} imply that for any $\sigma$,
\begin{equation}
 \label{eq:union is V}
\displaystyle V=\bigcup_{m\in M} E^\sigma_m,
\end{equation}
which was stated instead of $(ii)_f$ in Perling's original definition. To ease notations, we will write $(E^\sigma_m)$ for $(E^\sigma_m)_{\sigma\in \Sigma, m\in M}$.

A morphism between two such families of filtrations $(E^\sigma_m)$ and $(F^\sigma_m)$ of $V$ and $V'$ is a linear map $V\to V'$ that is compatible with the filtrations, namely sends $E^\sigma_m$ to $F^\sigma_m$ for each $\sigma\in\Sigma$ and $m\in M$. From Perling \cite[Theorem 5.18]{Per04}, the above description that sends $\Gamma(U_\sigma,\cE)_m$ to $E_\sigma^m$ induces an equivalence between the categories of equivariant torsion-free sheaves over $\C\P^n$ and families of multifiltrations of complex vector spaces. Axiom $(i)_f$ provides an equivariant quasi-coherent sheaf, that is coherent thanks to $(iii)_f$ and $(iv)_f$. Axiom $(v)_f$ gives torsion freeness, and $(ii)_f$ ensures that the rank of the sheaf is $r=\dim(V)$ (we shall refer to this as the rank of the family of filtrations as well).

Up to isomorphism, we can, and we will, assume that $V=\C^r$ in the families of multifiltrations that we will consider. From this point of view, an injective and equivariant morphism between two equivariant torsion-free sheaves of rank $r$ is equivalent to the data of two families of filtrations $(E^\sigma_m)$ and $(F^\sigma_m)$ (of $\C^r$) such that $E^\sigma_m\subset F^\sigma_m$ for all $\sigma\in\Sigma$ and all $m\in M$. We will denote this, somehow sloppily, as $(E^\sigma_m)\subset(F^\sigma_m)$.

We can then identify the reflexive hull of such sheaves thanks to the following lemma.

\begin{lemma}
 \label{lem:reflexive hull}
 Let $(E^\sigma_m)$ be a family of multifiltrations. Then the associated equivariant sheaf $\cE$ is reflexive if and only if for each $\sigma\in\Sigma$ and each $m\in M$,
 $$
 E^\sigma_m=\bigcap_{\rho\in\sigma(1)} E^\rho_m.
 $$
 In any case, the family of filtrations $(F^\rho(\bullet))_{\rho\in\Sigma(1)}$ for $\cF:=(\cE^\vee)^\vee$ is given, for $\rho\in\Sigma(1)$ and $i\in\Z$, by
 \begin{equation}
  \label{eq:reflexive hull}
 F^\rho(i)=E^\rho_m
 \end{equation}
 for any $m\in M$ such that $\langle m , u_\rho \rangle =i$.
\end{lemma}
The proof shall be clear from Perling's \cite{Per04} and the previous descriptions of toric and torsion-free equivariant sheaves, but we will include it for the sake of completeness. Let's first gather some useful facts that will be used freely in the proof and the following sections. Regarding the relations $(\leq_\sigma)_{\sigma\in\Sigma}$ on $M$, note that
\begin{itemize}
 \item[(i)] For each $\sigma$, $\leq_\sigma$ is transitive,
 \item[(ii)] For each $\sigma$,  $m\leq_\sigma m'$ and $m'\leq_\sigma m$ if and only if $m-m'\in\sigma^\perp$,
 \item[(iii)] For each $\tau\leq\sigma$, if $m\leq_\sigma m'$, then $m\leq_\tau m'$.
\end{itemize}
Note also that for a given family of multifiltrations $(E^\sigma_m)$ we have, for each $\tau\leq \sigma$ and $m\leq_\sigma m'$, inclusions $E^\sigma_m \subset E^\tau_m$ and $E^\sigma_m\subset E^\sigma_{m'}$ (cf. \eqref{eq:injective restrictions} and \eqref{eq:injective weight multiplication} that follow respectively from axioms $(v)_f$ and $(i)_f$ in Definition \ref{def:perling}).

\begin{proof}
 First, by point $(i)_f$, we see that if $m-m'\in\sigma^\perp$, then $E^\sigma_m=E^\sigma_{m'}$. Hence, the family of filtrations $(F^\rho(\bullet))_{\rho\in\Sigma(1)}$ defined by \eqref{eq:reflexive hull} is well defined. Set $\cF$ to be the associated equivariant reflexive sheaf. Then, as $\rho\leq \sigma$ implies $E^\sigma_m\subset E^\rho_m$, we have
 $$E^\sigma_m\subset \bigcap_{\rho\in\sigma(1)} E^\rho_m=\Gamma(U_\sigma,\cF)_m$$ for each $\sigma\in\Sigma$ and each $m\in M$, so that $\cE$ is an equivariant subsheaf of $\cF$. If the equality holds for each $\sigma\in\Sigma$ and $m\in M$, then $\cE=\cF$ and $\cE$ is reflexive. On the other hand, by construction of $(F^\rho(\bullet))$, $\cE$ and $\cF$ agree on each affine sets $U_\rho$, $\rho\in\Sigma(1)$, and on $U_{\lbrace 0\rbrace}$, and so do their reflexive hulls. As the complement of those sets is of dimension less or equal to $2$, they have the same reflexive hull on the whole $\C\P^n$ by normality of reflexive sheaves. We conclude that $(\cE^\vee)^\vee=\cF$, which ends the proof.
\end{proof}

\subsection{Elementary injections}
\label{sec:elementary injections}
We would like now to factorize any injection $\cE\to \cF$ between equivariant torsion-free sheaves of the same rank into a product of injections $\cE=\cE_0\to\cE_1\to\ldots\to\cE_p=\cF$ of the simplest possible form. 
\begin{definition}
 \label{def:elementary injection}
 Let $(E^\sigma_m)\subset (F^\sigma_m)$ be two families of multifiltrations of the same rank. We will say that it is {\it elementary}  if there is $0\leq k_0\leq n$ such that :
 \begin{enumerate}
  \item[$(i)_e$] For any $k<k_0$, any $\sigma\in\Sigma(k)$, and any $m\in M$, $E^\sigma_m=F^\sigma_m$,
  \item[$(ii)_e$] There is a unique couple $\sigma_0\in\Sigma(k_0)$ and $m_0\in M/(\sigma_0^\perp\cap M)$ such that $E^{\sigma_0}_{m_0}\neq F^{\sigma_0}_{m_0}$. For this couple, $\dim(F^{\sigma_0}_{m_0})=\dim(E^{\sigma_0}_{m_0})+1$. 
  \item[$(iii)_e$] For any $k>k_0$, any $\sigma\in\Sigma(k)$ and any $m\in M$, $E^\sigma_m=F^\sigma_m$ unless $\sigma_0<\sigma$ and $m\leq_{\sigma_0} m_0$ in which case $E^\sigma_m=F^\sigma_m\cap E^{\sigma_0}_{m_0}$.
  \end{enumerate}
  With the above notation, we will say that the injection is {\it $k_0$-elementary}, and refer to the triple $(k_0,\sigma_0,m_0)$ (or sometimes simply to the couple $(\sigma_0,m_0)$) as the parameters of the elementary injection. An injection between equivariant torsion-free sheaves is {\it elementary} if the associated injection of families of multifiltrations is.
  \end{definition}

To clarify this definition, recall that from \eqref{eq:injective weight multiplication}, for each cone $\sigma$ and for any $(m,m')\in M$ with $m'\in \sigma^\perp$, $ E^{\sigma}_{m}=E^{\sigma}_{m+m'}$ (cf. inclusion \eqref{eq:injective weight multiplication}). Point $(ii)_e$ then says that for any $\sigma\in\Sigma(k_0)\setminus\lbrace \sigma_0\rbrace$ and any $m$, $E^\sigma_m=F^\sigma_m$, and that for $\sigma_0$, there is a unique class $m\in M/(\sigma_0^\perp\cap M)$ where the equality fails. Point $(iii)_e$ makes sense, as the inequality $m\leq_{\sigma_0} m_0$ doesn't depend on the choice of the representant of the class of $m_0$ in $M/(\sigma_0^\perp\cap M)$. While points $(i)_e$ and $(ii)_e$ in the definition are quite clear, point $(iii)_e$ has some interesting consequences that we gather in the following lemmas.
\begin{lemma}
 \label{lem:dim diff elementary injection}
 Let $(E^\sigma_m)\subset (F^\sigma_m)$ be an elementary injection, and let $\sigma\in\Sigma$ and $m\in M$. Then either $\dim(F^\sigma_m)=\dim(E^\sigma_m)$ 
or $\dim(F^\sigma_m)=\dim(E^\sigma_m)+1$.
\end{lemma}

\begin{proof}
From Definition \ref{def:elementary injection}, $F^\sigma_m=E^\sigma_m$ unless maybe when   $\sigma_0\leq \sigma$ and $m\leq_{\sigma_0} m_0$. Note that the case $\sigma=\sigma_0$ is clear from $(ii)_e$, so we may assume $\sigma_0<\sigma$. In the latter case, $E^\sigma_m=F^\sigma_m\cap E^{\sigma_0}_{m_0}$. From $\sigma_0\leq \sigma$ and \eqref{eq:injective restrictions}, we have $F^\sigma_m\subset F^{\sigma_0}_m$, and from $m\leq_{\sigma_0} m_0$, we have $F^{\sigma_0}_{m}\subset F^{\sigma_0}_{m_0}$, hence 
$$F^{\sigma}_{m}\subset F^{\sigma_0}_{m_0}.$$
By $(ii)_e$, there is a line $L$ in $F^{\sigma_0}_{m_0}$ such that
$$
F^{\sigma_0}_{m_0}=E^{\sigma_0}_{m_0}\oplus L.
$$
Then, either $F^\sigma_m\subset E^{\sigma_0}_{m_0}$ in which case 
$$
E^\sigma_m=F^\sigma_m\cap E^{\sigma_0}_{m_0}=F^\sigma_m,
$$
or $F^\sigma_m\nsubseteq E^{\sigma_0}_{m_0}$ which implies
$$
F^\sigma_m=(F^\sigma_m\cap E^{\sigma_0}_{m_0})\oplus L'=E^\sigma_m\oplus L'
$$
for $L'$ a line directed by any element in $F^\sigma_m\setminus F^\sigma_m\cap E^{\sigma_0}_{m_0}$. The result follows.
\end{proof}
The following observation will also be useful (we keep notations from Definition \ref{def:elementary injection} in the statement).
\begin{lemma}
 \label{lem:precise when dim diff}
 Let $(E^\sigma_m)\subset (F^\sigma_m)$ be an elementary injection, and let $\sigma\in\Sigma$ and $m\in M$ be such that $\dim(F^\sigma_m)=\dim(E^\sigma_m)+1$. Then $\sigma_0\leq \sigma$ and $m-m_0\in\sigma_0^\perp$. Moreover, for any $m'\in m_0+\sigma_0^\perp$, if $m\leq_\sigma m'$, then  $\dim(F^\sigma_{m'})=\dim(E^\sigma_{m'})+1$.
\end{lemma}
\begin{proof}
 From Definition \ref{def:elementary injection}, we must have $\sigma_0\leq \sigma$  and $m\leq_{\sigma_0} m_0$. The case when $\sigma=\sigma_0$ follows easily from $(ii)_e$, so we assume $\sigma_0<\sigma$. We are in the same context as in the proof of Lemma \ref{lem:dim diff elementary injection}, and there is $x\in  F^\sigma_m\setminus F^\sigma_m\cap E^{\sigma_0}_{m_0}$. Then, from $\sigma_0\leq \sigma$   we deduce that $F^\sigma_m\subset F^{\sigma_0}_{m}$ so $x\in  F^{\sigma_0}_m\setminus F^{\sigma_0}_m\cap E^{\sigma_0}_{m_0}$. Hence $x\notin E^{\sigma_0}_{m_0}$, and then $x\notin E^{\sigma_0}_{m}$ as $E^{\sigma_0}_{m}\subset E^{\sigma_0}_{m_0}$ from $m\leq_{\sigma_0} m_0$. We conclude that $E^{\sigma_0}_{m}\neq F^{\sigma_0}_{m}$. From $(ii)_e$, $m-m_0\in\sigma_0^\perp$. This proves the first statement of the lemma.
 
 For the second statement, if $m'\in m_0+\sigma_0^\perp$, still assuming $\sigma_0<\sigma$, we are again in the setting of the proof of Lemma \ref{lem:dim diff elementary injection}. As $F^\sigma_m\nsubseteq E^{\sigma_0}_{m_0}$, from $m\leq_\sigma m'$ we deduce $F^\sigma_{m'}\nsubseteq E^{\sigma_0}_{m_0}$. The proof of the previous lemma then shows that $\dim(F^\sigma_{m'})=\dim(E^\sigma_{m'})+1$.
\end{proof}

We now prove the main result of this section.
\begin{theorem}
 \label{theo:existence decomposition elementary injections}
 Let $(E^\sigma_m)\subset (F^\sigma_m)$ be two families of filtrations of the same rank and $\cE\to\cF$ the corresponding injection of torsion-free sheaves. Then there exists a finite number of elementary injections
 $$
 \cE=\cE_0\to\cE_1\to\ldots\to\cE_p=\cF
 $$
 that factorize $\cE\to\cF$.
\end{theorem}
\begin{remark}
 \label{rem:decomposition I well ordered}
 In the proof of Theorem \ref{theo:existence decomposition elementary injections}, we will see that in the factorization
  $$
 \cE=\cE_0\to\cE_1\to\ldots\to\cE_p=\cF
 $$
 if $\cE_i\subset\cE_{i+1}$ is $k_i$-elementary, then $i\mapsto k_i$ is decreasing. Note that the proof could be adapted to have instead $i\mapsto k_i$ increasing.
\end{remark}

\begin{proof}
 We shall introduce an invariant $\delta=(\delta_k)_{1\leq k\leq n}\in(\N\cup\lbrace +\infty\rbrace)^n$ for injections between families of filtrations of the same rank such that $\delta=0$ if and only if the injection is an isomorphism, and then proceed to the proof by some induction process. Let $(E^\sigma_m)\subset (F^\sigma_m)$ be two families of filtrations of the same rank (for the same vector space, say $\C^r$). We define $\delta$ inductively on the dimension of the cones of $\Sigma$. First, for each ray $\rho\in\Sigma(1)$, we set
 $$
 \delta_1^\rho:=\sum_{m\in M/(\rho^\perp\cap M)}( \dim(F^\rho_m)-\dim(E^\rho_m))\in \N.
 $$
 This sum is indeed finite, as from point $(iii)_f$ of Definition \ref{def:perling} and by \eqref{eq:union is V}, there are $(i,j)\in\Z^2$ such that $E^\rho_m=0$ if $\langle m, u_\rho\rangle \leq i$ and $E^\rho_m=\C^r$ if $\langle m, u_\rho\rangle \geq j$, and similarly for the $F^\rho_m$'s. Hence, for $\langle m, u_\rho\rangle$ small enough, $E^\rho_m=F^\rho_m=0$ while for $\langle m, u_\rho\rangle$ large enough, $E^\rho_m=F^\rho_m=\C^r$. Then we set 
 $$
 \delta_1:=\sum_{\rho\in\Sigma(1)} \delta_1^\rho\in\N.
 $$
 Note that by construction, $\delta_1=0$ if and only if for any $\rho\in\Sigma(1)$, $E_m^\rho=F_m^\rho$ for all $m\in M$. We proceed to two dimensional cones in the following way. For all $\sigma\in\Sigma(2)$, we have a decomposition into rays $\sigma=\rho_1+\rho_2$. Consider the subset $\Sigma^*(2)\subset\Sigma(2)$ of cones $\sigma=\rho_1+\rho_2$ such that $\delta_1^{\rho_1}=\delta_1^{\rho_2}=0$. If this set is empty, then set $\delta_k=+\infty$ for $k\geq 2$, which ends the construction of $\delta$. Note in this case $(E_m^\sigma)\neq(F_m^\sigma)$. If it's not empty, for any $\sigma=\rho_1+\rho_2\in \Sigma^*(2)$ we set 
 $$
 \delta_2^\sigma:=\sum_{m\in M/(\sigma^\perp\cap M)} (\dim(F^\sigma_m)-\dim(E^\sigma_m))\in \N.
 $$
 Finiteness of this sum deserves some explanations. Pick a basis $\lbrace e_1,\ldots,e_n\rbrace$ of $N$ with dual basis $\lbrace e_1^*,\ldots,e_n^*\rbrace$ for $M$ so that $\sigma=\R_+\cdot e_1 + \R_+\cdot e_2$ and
 $$
 M/(\sigma^\perp\cap M)\simeq \Z\cdot e_1^*+\Z\cdot e_2^*.
 $$
 From $(iii)_f$, there are integers $a_0$ and $b_0$ such that $$F^\sigma_{a e_1^*+b e_2^*}=E^\sigma_{a e_1^*+b e_2^*} =0$$ unless $a\geq a_0$ and $b\geq b_0$. Hence
 $$
 \delta_2^\sigma=\sum_{ a\geq a_0, b\geq b_0}( \dim(F^\sigma_{a e_1^* + b e_2^*})-\dim(E^\sigma_{a e_1^* + b e_2^*})).
 $$
 As $\delta_1^{\rho_1}=0$, there are integers $a_0 < a_1 <\ldots< a_p$ such that the dimension of the spaces $E^{\rho_1}_{a e_1^*}=F^{\rho_1}_{a e_1^*}$ is constant when $a$ varies in each set $\lbrace a_0,a_0+1,\ldots,a_1-1\rbrace$, $\lbrace a_1,a_1+1,\ldots,a_2-1\rbrace$, $\ldots$, $\lbrace a_{p-1},\ldots, a_p-1\rbrace$ and $\lbrace a_p, a_p+1,\ldots \rbrace$. We thus sum
 $$
 \delta_2^\sigma=\sum_{i=0}^p\left(\sum_{ a_i\leq a< a_{i+1}}\sum_{ b\geq b_0} (\dim(F^\sigma_{a e_1^* + b e_2^*})-\dim(E^\sigma_{a e_1^* + b e_2^*}))\right),
 $$
 where we set $a_{p+1}=+\infty$. For each $i\leq p-1$, the sum 
 $$
 \sum_{ a_i\leq a< a_{i+1}}\sum_{b\geq b_0} (\dim(F^\sigma_{a e_1^* + b e_2^*})-\dim(E^\sigma_{a e_1^* + b e_2^*}))
 $$
 is finite, as from $(v)_f$, for each $a$, there is $b_a$ such that for $b\geq b_a$, $E^\sigma_{a e_1^*+b e_2^*}=E^{\rho_1}_{a e_1^*}=F^{\rho_1}_{a e_1^*}=F^\sigma_{a e_1^*+b e_2^*}$. It remains to show that 
 $$
 \sum_{ a_p\leq a}\sum_{b_0\leq b} (\dim(F^\sigma_{a e_1^* + b e_2^*})-\dim(E^\sigma_{a e_1^* + b e_2^*}))
 $$
 is finite. Using again $(v)_f$, we have for each $a\geq a_p$ a smallest integer $b_a^j$ such that for $b\geq b_a^1$, $E^\sigma_{a e_1^*+b e_2^*}=E^{\rho_1}_{a e_1^*}=\C^r$ and for $b\geq b_a^2$, $F^\sigma_{a e_1^*+b e_2^*}=F^{\rho_1}_{a e_1^*}=\C^r$. Moreover, by $(i)_f$, we see that $a\mapsto b_a^j$ must be decreasing, reaching a minimum in finite time due to $(iv)_f$. Hence, there are $a'\in \Z$ and $b'\in\Z$ such that for $a\geq a'$ and $b\geq b'$, $E^\sigma_{a e_1^*+b e_2^*}=\C^r=F^\sigma_{a e_1^*+b e_2^*}$. We hence reduced ourselves to showing finiteness of
 $$
 \sum_{ a'\leq a}\sum_{b_0\leq b\leq b'} (\dim(F^\sigma_{a e_1^* + b e_2^*})-\dim(E^\sigma_{a e_1^* + b e_2^*}))
 $$
 or equivalently
 $$
 \sum_{b_0\leq b\leq b'} \sum_{ a'\leq a}(\dim(F^\sigma_{a e_1^* + b e_2^*})-\dim(E^\sigma_{a e_1^* + b e_2^*})),
 $$
 which can be achieved by similar arguments, projecting now onto the $e_2^*$ coordinate and finding $a_b$'s such that for $a\geq a_b$, $E^\sigma_{a e_1^*+b e_2^*}=E^{\rho_2}_{b e_2^*}=F^{\rho_2}_{b e_2^*}=F^\sigma_{a e_1^*+b e_2^*}$.
 
 We then set 
 $$
 \delta_2:=\sum_{\sigma\in\Sigma^*(2)} \delta_2^\sigma\in\N.
 $$
 This time, by construction, $(\delta_1,\delta_2)=(0,0)$ if and only if $E_m^\sigma=F_m^\sigma$ for all $\sigma\in\Sigma(1)\cup\Sigma(2)$ and all $m\in M$ (note that the vanishing of $\delta_1$ implies $\Sigma^*(2)=\Sigma(2)$). We then proceed inductively, considering for each $2\leq k\leq n$, the set $\Sigma^*(k)$ of $k$-dimensional cones $\sigma$ such that for each facet $\tau < \sigma$, $\tau\in\Sigma^*(k-1)$ and $\delta_{k-1}^\tau=0$. When $\Sigma^*(k)$ is empty we set $\delta_j=+\infty$ for $j\geq k$, so that $\delta=(\delta_1,\ldots,\delta_{k-1},+\infty,\ldots,+\infty)$. Otherwise we can then set for $\sigma\in\Sigma^*(k)$
 $$
 \delta_k^\sigma:=\sum_{m\in M/(\sigma^\perp\cap M)} (\dim(F^\sigma_m)-\dim(E^\sigma_m)),
 $$
 which is again finite by a similar argument, and we define
 $$
 \delta_k=\sum_{\sigma\in\Sigma^*(k)} \delta_k^\sigma.
 $$
 This time, vanishing of $(\delta_1,\delta_2,\ldots,\delta_k)$ implies equality for the associated torsion-free sheaves on the affine sets $U_\sigma$ for any $\sigma$ of dimension less or equal to $k$. In the end, we set
 $$
 \delta=(\delta_1,\ldots,\delta_n)
 $$
 which is our desired invariant. 
 
 We now proceed to our proof, and show that any injection $(E^\sigma_m)\subset (F^\sigma_m)$ of families of filtrations of the same rank can be factored through elementary injections. We proceed by descending induction on the smallest $k\in\lbrace 1,\ldots,n\rbrace$ such that $\delta_k\neq 0$. Note that by construction, if $\delta=(0,\ldots,0,\delta_k,\ldots,\delta_n)$, then $\Sigma(k)=\Sigma^*(k)$ and $\delta_k$ is finite. If this smallest integer is $n$, we proceed by induction on $\delta_n$ (note that the case $\delta_n=0$ as well corresponds to an equality of torsion-free sheaves, and there is nothing to prove here). Note also that $\delta=(0,\ldots,0,\delta_n)$ implies $E^\sigma_m=F^\sigma_m$ for $\sigma\in\Sigma$ of dimension less or equal to $n-1$. If $\delta_n=1$, then by construction of $\delta$, there is a unique cone $\sigma\in\Sigma(n)$ and a unique weight $m\in M$ such that 
 $$
 \dim(F^\sigma_m)=\dim(E^\sigma_m)+1
 $$
 and for all the other pairs of a cone $\sigma'$ in $\Sigma$ and a weights $m'\in M$, the inclusions $E^{\sigma'}_{m'}\subset F^{\sigma'}_{m'}$ are equalities. By Definition \ref{def:elementary injection}, $(E^\sigma_m)\subset (F^\sigma_m)$ is already an elementary injection, which ends the proof in that case. If the result is true for $\delta_n\leq d$, let's assume $\delta_n=d+1$. Pick $\sigma_0\in\Sigma(n)$ and $m_0\in M$ such that $E^{\sigma_0}_{m_0}\subsetneq F^{\sigma_0}_{m_0}$. We may assume that $m_0$ is chosen such that 
 \begin{equation}
  \label{eq:choice m0 for G}
  \forall\; m<_\sigma m_0\;,\: E^{\sigma_0}_{m}= F^{\sigma_0}_{m}.
 \end{equation}
Indeed, from the construction of $\delta_{n}$, there is a finite number of weights $m\in M$ such that $E^{\sigma_0}_{m}\neq F^{\sigma_0}_{m}$. Then, define a family of multifiltrations $(G^\sigma_m)$ in the following way :
 \begin{enumerate}
  \item[$(G1)$] For any $k<n$, any $\sigma\in\Sigma(k)$, and any $m\in M$, $G^\sigma_m=E^\sigma_m=F^\sigma_m$.
  \item[$(G2)$] For $\sigma_0$ and $m_0$, $E^{\sigma_0}_{m_0}\subset G^{\sigma_0}_{m_0}\subsetneq F^{\sigma_0}_{m_0}$. For this couple, $\dim(G^{\sigma_0}_{m_0})+1=\dim(F^{\sigma_0}_{m_0})$.
  \item[$(G3)$] For $\sigma_0$ and $m\in M\setminus\lbrace m_0\rbrace$, $G_m^\sigma=F_m^\sigma$.
  \item[$(G4)$] For any $\sigma\in\Sigma(n)\setminus\lbrace\sigma_0\rbrace$, $G^\sigma_m=F^\sigma_m$ for all $m\in M$. 
 \end{enumerate}
 
We claim that $(G^\sigma_m)$ is indeed a family of multifiltrations of rank $r$. The axioms $(ii)_f, (iii)_f$ and $(iv)_f$ from Definition \ref{def:perling} are clearly satisfied as they are valid for $(F^\sigma_m)$. Similarly, the compatibility conditions $(i)_f$ and $(v)_f$ are clearly verified for cones different from $\sigma_0$ as $G^\sigma_m=F^\sigma_m$ in that case. For $\sigma_0$, axiom $(i)_f$ is valid as if $m\leq_{\sigma_0} m_0$, either $m=m_0$ and $G_m^{\sigma_0}=G_{m_0}^{\sigma_0}$, or $m<_{\sigma_0} m_0$ and  then 
$$
\begin{array}{ccccc}
 G^{\sigma_0}_m & = &   F^{\sigma_0}_m & \mathrm{ by } & (G3),\\
                & = & E^{\sigma_0}_m & \mathrm{ by }&  \eqref{eq:choice m0 for G},\\
                & \subset & E^{\sigma_0}_{m_0}&  \mathrm{ as }& m\leq_{\sigma_0} m_0,\\
                & \subset & G_{m_0}^{\sigma_0}& \mathrm{ by }& (G2).
\end{array}
$$
If instead $m_0\leq_{\sigma_0} m$, with $m\neq m_0$, then
$$
\begin{array}{ccccc}
 G^{\sigma_0}_{m_0} & \subset &   F^{\sigma_0}_{m_0} & \mathrm{ by } & (G2),\\
                & \subset & F^{\sigma_0}_m &  \mathrm{ as }& m_0\leq_{\sigma_0} m,\\
                & \subset & G^{\sigma_0}_{m}&  \mathrm{ by }& (G3).
\end{array}
$$
In general, if we consider two weights $(m,m')\in M^2$ with $m\leq_{\sigma_0} m'$, either one of the weights equals $m_0$ and then by the previous discussion we have $G_m^{\sigma_0}\subset G_{m'}^{\sigma_0}$, or both are different from $m_0$ and the same conclusion holds using $(G3)$. Finally, $(v)_f$ is satisfied for $(G^\sigma_m)$ because it is for $(F^\sigma_m)$ and because the chain $i\mapsto G^{\sigma_0}_{m+im_\tau}$ equals $i\mapsto F^{\sigma_0}_{m+im_\tau}$ for $i$ large enough. Hence the claim is true and $(G^\sigma_m)$ is a family of multifiltrations. 

It is clear that $(E^\sigma_m)\subset (G^\sigma_m)\subset (F^\sigma_m)$ and that $(G^\sigma_m)\subset (F_m^\sigma)$ is an elementary injection. Last, by construction, the $\delta_n$ invariant of $(E^\sigma_m)\subset(G^\sigma_m)$ is $d$, hence we may apply our induction hypothesis to factorize this injection into elementary ones, which ends the proof in that case. 

We assume now that the result is true for any injection with $\delta$ invariant of the form $(0,\ldots,0,\delta_{k_0+1},\ldots,\delta_n)$, and will give the proof when the smallest $k$ such that $\delta_k\neq 0$ is $k_0$ (note that when one reaches $\delta_k=0$ in the induction, this automatically settles $\delta_{k+1}$ to be finite, hence our induction scheme is well defined). We again proceed in that case by induction on $\delta_{k_0}\in\N$, and we assume that $\delta_{k_0}>0$. We will decompose as before
\begin{equation}
 \label{eq:factorize proof}
(E_m^\sigma)\subset (H_m^\sigma)\subset (F_m^\sigma)
\end{equation}
such that the second inclusion is elementary, and the $\delta_{k_0}$ invariant of the first inclusion has decreased by one. If this new $\delta_{k_0}$ invariant reaches $0$, we can apply our induction hypothesis on the smallest $k$ such that $\delta_k\neq 0$, and if not, we apply our induction hypothesis on $\delta_{k_0}$ to conclude.

We are left with the problem of constructing the above factorization \eqref{eq:factorize proof}. By construction of $\delta_{k_0}$, we can find $\sigma_0\in\Sigma(k_0)$ and $m_0\in M/(\sigma_0^\perp\cap M)$ such that $E^{\sigma_0}_{m_0}\subsetneq F^{\sigma_0}_{m_0}$. As in the case for $\delta_n$, arguing this time on the classes of weights $m\in M/(\sigma_0^\perp\cap M)$, we may assume that $m_0$ is chosen such that
\begin{equation}
 \label{eq:choice m0 for H}
\forall\; m<_{\sigma_0} m_0\;,\: E^{\sigma_0}_{m}= F^{\sigma_0}_{m}. 
\end{equation}
Then, we introduce the family of multifiltrations $(H_m^\sigma)$ given by :
 \begin{enumerate}
  \item[$(H1)$] For any $k<k_0$, any $\sigma\in\Sigma(k)$, and any $m\in M$, $H^\sigma_m=E^\sigma_m=F^\sigma_m$.
  \item[$(H2)$] For $\sigma_0\in\Sigma(k_0)$ and $m_0\in M/(\sigma_0^\perp\cap M)$, $E^{\sigma_0}_{m_0}\subset H^{\sigma_0}_{m_0}\subsetneq F^{\sigma_0}_{m_0}$. For this couple, $H^{\sigma_0}_{m_0}=H$ for $H$ a hyperplane containing $E^{\sigma_0}_{m_0}$ in $F^{\sigma_0}_{m_0}$.
  \item[$(H3)$] For $\sigma_0$ and $m\in M/(\sigma_0^\perp\cap M)$ such that $m- m_0\notin \sigma_0^\perp$, $H_m^{\sigma_0}=F_m^{\sigma_0}$.
  \item[$(H4)$] For $\sigma\in\Sigma(k_0)\setminus\lbrace \sigma_0\rbrace$,  $H^\sigma_m=F^\sigma_m$ for all $m\in M$. 
  \item[$(H5)$] For any $k>k_0$, any $\sigma\in\Sigma(k)$ and any $m\in M$, $H^\sigma_m=F^\sigma_m$ unless $\sigma_0<\sigma$ and $m\leq_{\sigma_0} m_0$ in which case $H^\sigma_m=F^\sigma_m\cap H^{\sigma_0}_{m_0}$.
 \end{enumerate}
 
 Note that  in $(H2)-(H3)$ we identified weights with their classes modulo $(\sigma_0^\perp\cap M)$ with no harm, as translations by elements $m\in\sigma_0^\perp$ induce equalities $F^{\sigma_0}_\bullet=F^{\sigma_0}_{\bullet+m}$. Note also that in $(H5)$, the relation $m\leq_{\sigma_0} m_0$ is well defined and doesn't depend on the choice of a representant of the class $m_0\in M/(\sigma_0^\perp\cap M)$.
 
 First, we note that we have for any cone $\sigma$ and any $m\in M$, 
 $$
 E^\sigma_m\subset H^\sigma_m\subset F^\sigma_m.
 $$
 Indeed, the second inclusion is clear while the first  follows from $(E^\sigma_m)\subset (F^\sigma_m)$, except maybe when $\sigma_0\leq \sigma$ and $m\leq_{\sigma_0} m_0$. In this latter case, $E^\sigma_m\subset E_m^{\sigma_0}\subset E_{m_0}^{\sigma_0}$ by \eqref{eq:injective restrictions} and \eqref{eq:injective weight multiplication}; this in turn implies that $E^\sigma_m\subset H^\sigma_m$ by $(H2)$.
 
 We now show that $(H^\sigma_m)$ is a family of multifiltrations, by checking the axioms in Definition \ref{def:perling}. First, $(ii)_f-(iii)_f-(iv)_f$ easily follows  from the fact that $(F^\sigma_m)$ satisfies those axioms. We now show compatibility condition $(i)_f$. Let $\sigma\in\Sigma$. If $\sigma=\sigma_0$, we can argue as we did before for $(G^\sigma_m)$ to conclude.  If $\sigma_0$ is not a face of $\sigma$, then $(i)_f$ follows as $H^\sigma_m=F^\sigma_m$ for all weights $m$. Finally, if $\sigma_0<\sigma$, let $m\leq_\sigma m'$. Recall that this implies in particular $m\leq_{\sigma_0} m'$. Then, a case by case analysis on which weights $\tilde m\in\lbrace m,m'\rbrace$ satisfy $\tilde m\leq_{\sigma_0} m_0$ leads to the result (note that by transitivity of $\leq_{\sigma_0}$, the case  $m'\leq_{\sigma_0} m_0$ forces also $m\leq_{\sigma_0}m_0$). It remains to check axiom $(v)_f$. The only non-trivial cases are for cones $\sigma$ with $\sigma_0\leq \sigma$. If $\sigma=\sigma_0$, the proof goes as for $(G^\sigma_m)$. If $\sigma_0 < \sigma$, consider a facet $\tau <\sigma$ and fix $m_\tau\in \sigma^\vee\cap\tau^\perp$ such that $\tau^\perp\cap M=\tau^\perp\cap M+\N\cdot (-m_\tau)$. Fix then $m\in M$ and consider the sequence $(H^\sigma_{m+im_\tau})_{i\in\N}$. We want to show that it is stationary at $H^\tau_m$. If $\sigma_0\leq \tau$, then $m_\tau\in \sigma_0^\perp$ so either $m+im_\tau\leq_{\sigma_0} m_0$ for all $i\in\N$, or for none. In both cases, we can conclude using $(H5)$ and the fact that  $(F^\sigma_{m+im_\tau})_{i\in\N}$ is stationary at $F^\tau_m$. If $\sigma_0\nleq \tau$, then from $\sigma_0 < \sigma$ we deduce that $m_\tau\in\sigma_0^\vee$. But then, for $i$ large enough, $m+im_\tau\nleq_{\sigma_0} m_0$, so that $H^\sigma_{m+im_\tau}=F^\sigma_{m+im_\tau}$ for large $i$, and $(H^\sigma_{m+im_\tau})_{i\in\N}$ is stationary at $F^\tau_m$, which equals $H^\tau_m$ by $(H5)$ and the fact that $\sigma_0$ is not a face of $\tau$.  This concludes the proof that $(H^\sigma_m)$ is a family of multifiltrations.

 Finally, by construction $(H^\sigma_m)\subset (F^\sigma_m)$ is elementary, and $(E^\sigma_m)\subset (H^\sigma_m)$ has $\delta_{k_0}$-invariant strictly smaller than the one of $(E^\sigma_m)\subset (F^\sigma_m)$,  so we can apply our induction hypothesis, and conclude the proof.
\end{proof}

For later use, we will introduce a specific class of elementary injections. Let $(E^\sigma_m)\subset (F^\sigma_m)$ be an elementary injection with parameters $(k_0,\sigma_0,m_0)$. For each $\sigma\in \Sigma$, we introduce the following set of (classes of) weights :
\begin{equation}
 \label{eq:definition Wsigma}
W_\sigma:=\lbrace m\in M/(\sigma^\perp\cap M)\:\vert\: \dim(F^\sigma_m)=\dim(E^\sigma_m)+1\rbrace\subset M/(\sigma^\perp\cap M).
\end{equation}
Note that by Lemma \ref{lem:dim diff elementary injection}, if $m\notin W_\sigma$, then $F^\sigma_m=E^\sigma_m$, while by Lemma \ref{lem:precise when dim diff}, $W_\sigma\neq \emptyset$ if and only if $\sigma_0\leq \sigma$, in which case $W_\sigma$ is included in the preimage of $m_0$ under the projection map 
$$
M/(\sigma^\perp\cap M)\to M/(\sigma_0^\perp\cap M).
$$
  Let $\sigma\in\Sigma(d)$ be such that $\sigma_0 <\sigma$.  Assume  that $\sigma_0=\R_+e_1+\ldots+\R_+e_{k_0}$, while $\sigma=\R_+e_1+\ldots+\R_+e_d$, for $k_0\leq d$ and $(e_i)_{1\leq i\leq d}=(u_\rho)_{\rho\in\sigma(1)}\in N$ is part of a $\Z$-basis. We make the following identifications 
$$
M/(\sigma_0^\perp\cap M)\simeq \Z\cdot e_1^*\oplus\ldots\oplus\Z\cdot e_{k_0}^*
$$
and 
$$
M/(\sigma^\perp\cap M)\simeq \Z\cdot e_1^*\oplus\ldots\oplus\Z\cdot e_{d}^*.
$$
For any $j$ such that $k_0< j \leq d$, consider the cone $\sigma_j=\sigma_0+\R_+e_j$ that contains $\sigma_0$ as a facet.  From Axiom $(v)_f$ of Definition \ref{def:perling}, for $i$ large enough, we have 
$$
E^{\sigma_j}_{m_0+ie_j^*}=E^{\sigma_0}_{m_0}
$$
and 
$$
F^{\sigma_j}_{m_0+ie_j^*}=F^{\sigma_0}_{m_0}.
$$
Hence, from $(iii)_f$ in Definition \ref{def:perling}, $(ii)_e$ in Definition \ref{def:elementary injection} and from Lemma \ref{lem:precise when dim diff}, there exists $a_j\in\Z$ such that for any $i < a_j$, $\dim(F^{\sigma_j}_{m_0+ie_j^*})=\dim(E^{\sigma_j}_{m_0+ie_j^*})$ while for any $ i\geq a_j$, $\dim(F^{\sigma_j}_{m_0+ie_j^*})=\dim(E^{\sigma_j}_{m_0+ie_j^*})+1$. With those integers $(a_j)_{k_0< j\leq d}$ at hand, set 
\begin{equation}
 \label{eq:msigma}
m_\sigma:=m_0+a_{k_0+1}e_{k_0+1}^*+\ldots+a_d e_d^*\in M/(\sigma^\perp\cap M).
\end{equation}
Notice that $\langle m_\sigma,u_\rho\rangle=a_j$. Moreover we have :
\begin{lemma}
 \label{lem:localization diff dim}
 Assume that $m\in W_\sigma$. Then $m_\sigma\leq_\sigma m$.
\end{lemma}
\begin{proof}
As $m\in W_\sigma$, by Lemma \ref{lem:precise when dim diff}, $m-m_0\in\sigma_0^\perp$. We can then 
write 
$$
m=m_0+m_{k_0+1}e_{k_0+1}^*+\ldots+m_d e_d^*\in M/(\sigma^\perp\cap M).
$$
If $m_\sigma\nleq_\sigma m$, then there is $k_0<j\leq d$ such that $m_j < a_j$. From $\sigma_j\leq \sigma$, we have 
$F^\sigma_m\subset F^{\sigma_j}_m$. But then, as $m_j<a_j$, and by definition of $a_j$, we must have $F^{\sigma_j}_m=E^{\sigma_j}_m$. Then, as in the proof of Lemma \ref{lem:dim diff elementary injection}, this implies that $F^{\sigma_j}_m\subset E_{m_0}^{\sigma_0}$, hence $F^{\sigma}_m\subset E_{m_0}^{\sigma_0}$, which in turn forces $E^{\sigma}_m=F^\sigma_m$ (again following the proof of Lemma \ref{lem:dim diff elementary injection}). As $m\in W_\sigma$, this is absurd, which concludes the proof.
\end{proof}
The previous lemma then says that for any $\sigma\in\Sigma$ with $\sigma_0\leq \sigma$,
$$
W_\sigma\subset (m_0+\sigma_0^\perp)\cap (m_\sigma+\sigma^\vee).
$$
This suggests the following definition :
\begin{definition}
 \label{def:simple elementary injection}
 Let $(E^\sigma_m)\subset (F^\sigma_m)$ be an elementary  injection of families of multifiltrations with parameters $(\sigma_0,m_0)$. We will say that it is {\it saturated} if for any $\sigma_0\leq\sigma$, we have 
 $$
 W_\sigma= (m_0+\sigma_0^\perp)\cap (m_\sigma+\sigma^\vee).
 $$
 We will say that an elementary injection of equivariant torsion-free sheaves is saturated if the associated injection of families of multifiltrations is.
\end{definition}
We will now provide geometric interpretations of those notions.

\begin{proposition}
 \label{prop:elementary injection geometric}
  Let $\cE\to\cF$ be an equivariant injection between torsion-free equivariant sheaves of the same rank, and denote by $\cQ:=\cF/\cE$ the quotient. Then $\cQ$ is equivariant, pure with reduced irreducible support $V:=V(\sigma_0)$ and  restriction $\cQ_{\vert V}$ of rank $1$ if and only if $\cE\to\cF$ is elementary with parameters $(\sigma_0,m_0)$, for $(\sigma_0,m_0)\in\Sigma\times M$. If in addition $\cE\to\cF$ is saturated, then $\cQ_{\vert V}$ is a line bundle.
\end{proposition}
In the above statement, we used the notation $V(\sigma_0)$ to denote the closure of the orbit $O(\sigma_0)$ associated to $\sigma_0\in\Sigma$ in the orbit cone correspondence \cite[Section 3.2]{CLS}. 
\begin{proof}
Assume that $\cE\to\cF$ is elementary with parameters $(\sigma_0,m_0)$.
 Consider for each $\sigma\in\Sigma$ and each $m\in M$ : 
 $$
 Q^\sigma_m:=F^\sigma_m/E^\sigma_m.
 $$
 It is straightforward to check that $(Q^\sigma_m)$ satisfies the axioms of a finite $\Delta$-family in Perling's terminology \cite[Section 5]{Per04}. Such a family corresponds to an equivariant coherent sheaf, namely the quotient sheaf $\cQ$ (working on a smooth toric variety, the data of an equivariant pre-sheaf corresponds to the data of an equivariant sheaf, from the orbit-cone correspondence \cite[Section 3.2]{CLS}). By construction, we have for all $\sigma\in\Sigma$, $Q^\sigma_m=0$ if $\sigma_0\nleq \sigma$, while when $\sigma_0\leq\sigma$, $Q^\sigma_m=0$ unless maybe when $m\in W_\sigma$, in which case $Q^\sigma_m$ is a line. In this latter case, $m_\sigma\leq_\sigma m$ by Lemma \ref{lem:localization diff dim}. Those are precisely the conditions for $(Q^\sigma_m)$ to define a pure equivariant sheaf with reduced support $V(\sigma_0)$, and whose restriction to $V(\sigma_0)$ is of rank $1$, cf \cite[Proposition 2.8 and Theorem 2.10]{Koo11}.
 The converse follows from similar considerations, and we leave the details to the interested reader. 
 
 Assume now in addition that the injection is saturated. Then by construction, for each $\sigma_0\leq\sigma$, and for each $m\in W_\sigma$, writing $\sigma_0=\R_+e_1+\ldots+\R_+e_{k_0}$ and $\sigma=\R_+e_1+\ldots+\R_+e_d$ as before, we see that 
 $$
 Q^\sigma_m=\bigcap_{j=k_0+1}^d Q^{\sigma_j}_{m_0+m_je_j^*}.
 $$
 Then, from the description of the toric variety $V(\sigma_0)$ in \cite[Section 3.2]{CLS}, and by Lemma \ref{lem:reflexive hull}, we conclude that the restriction of $\cQ$ to $V(\sigma_0)$ is reflexive. As it is of rank $1$, it is locally free, which ends the proof.
 
\end{proof}

\begin{remark}
 Using similar arguments as in the proof of Theorem \ref{theo:existence decomposition elementary injections}, one may show that for any $k_0$-elementary injection $\cE\to\cF$, there is a finite number of $k_0$-elementary injections 
 $$
 \cE\to\cF\to\cF_1\to\ldots\to\cF_p
 $$
 such that $\cE\to \cF_p$ is $k_0$-elementary and saturated. The idea is to use an inductive argument on the cardinality of $(m_0+\sigma_0^\perp)\cap (m_\sigma+\sigma^\vee)\setminus W_\sigma$, for $\sigma$ of minimal dimension such that the difference of those sets is non-empty. Then, one splits such an injection by modifying the family of filtration at $E^\sigma_m$, for a suitable weight $m$. We leave the details of the proof of this fact to the interested reader, as we won't use it in what follows.
\end{remark}

Together with Proposition \ref{prop:elementary injection geometric}, Theorem \ref{theo:existence decomposition elementary injections} implies :
\begin{corollary}
\label{cor:splitting geometric}
 Let $\cE\to\cF$ be an equivariant injection of equivariant torsion-free sheaves of the same rank. Then it splits as 
 $$
 \cE=\cE_0\to\cE_1\to \ldots\to\cE_p=\cF,
 $$
 where for each $i$, $\cQ_i:=\cE_{i+1}/\cE_i$ is equivariant, pure and restricts to a rank $1$ sheaf on its reduced irreducible support. Moreover, the map $i\mapsto \dim(\mathrm{Supp}(\cQ_i))$ is increasing.
\end{corollary}

\section{Elementary injections and Chern classes}
\label{sec:chern classes elementary}
There is a formula for the total Chern class of a torsion-free equivariant sheaf over a smooth projective toric variety. To state it, we first need to introduce some notations. Let $(E^\sigma_m)$ be a family of multifiltrations corresponding to a torsion-free sheaf over some smooth projective toric variety. Fix a cone $\sigma\in\Sigma(d)$, that we may write $$\sigma=\R_+e_1+\ldots+\R_+e_d\subset N_\R$$ for $\lbrace e_1,\ldots,e_d\rbrace$ part of a $\Z$-basis of $N$. Then, following Kool's notation \cite{Koo11}, we may rewrite the multifiltration $(E^\sigma_m)_{m\in M}$ as $(E^\sigma(\lambda_1,\ldots,\lambda_d))_{(\lambda_1,\ldots,\lambda_d)\in\Z^d}$, where each $m\in M/(\sigma^\perp\cap M)$ writes uniquely (identifying $M/(\sigma^\perp\cap M)$ with $\Z\cdot e_1^*\oplus\ldots\oplus \Z \cdot e_d^*$)
$$
m=\lambda_1e_1^*+\ldots+\lambda_de_d^*.
$$
We then introduce for each $i\in\lbrace 1,\ldots,d\rbrace$ a $\Z$-linear operator $\Delta_i$ on the free abelian group generated by the vector spaces $\lbrace E^\sigma(\lambda_1,\ldots,\lambda_d)\rbrace_{(\lambda_1,\ldots,\lambda_d)\in\Z^d}$ defined by the following, for $(\lambda_1,\ldots,\lambda_d)\in\Z^d$ :
$$
\Delta_i \left(E^\sigma(\lambda_1,\ldots,\lambda_d)\right)=E^\sigma(\lambda_1,\ldots,\lambda_d)-E^\sigma(\lambda_1,\ldots,\lambda_{i-1},\lambda_i-1,\lambda_{i+1},\ldots,\lambda_d).
$$
We then set for any $(\lambda_1,\ldots,\lambda_d)\in\Z^d$ :
$$
[E^\sigma](\lambda_1,\ldots,\lambda_d)=\Delta_1(\ldots(\Delta_r\left(E^\sigma(\lambda_1,\ldots,\lambda_d) \right))\ldots).
$$
We then extend the definition of the dimension in a $\Z$-linear way on the free abelian group generated by the multifiltration, and we can consider the integer
$$
\dim\left( [E^\sigma](m)\right):=\dim\left( [E^\sigma](\lambda_1,\ldots,\lambda_d)\right)\in\Z,
$$
where we used the identification $m=\lambda_1e_1^*+\ldots+\lambda_de_d^*$. For example, for a ray $\rho\in\Sigma(1)$, and $m=\lambda_1e_1^*$, we have
$$
\dim([E^\rho](m)) = \dim( E^\rho(\lambda_1) )- \dim( E^\rho(\lambda_1 -1)),
$$
while for  $\sigma \in \Sigma(2)$, and $m=\lambda_1e_1^*+\lambda_2e_2^*$, we have
\begin{eqnarray*}
\dim( [E^\sigma](m) ) =  &\dim( E^\sigma(\lambda_1, \lambda_2))\\
 & - \dim( E^\sigma(\lambda_1-1, \lambda_2)) - \dim (E^\sigma(\lambda_1, \lambda_2-1))
\\ 
&+ \dim( E^\sigma(\lambda_1-1, \lambda_2-1)).
\end{eqnarray*}
In general, the formula   is given by
\begin{equation}
 \label{eq:dimension Esigma crochet}
 \dim( [E^\sigma](m) ) =  \sum_{\mu\in\lbrace 0, 1 \rbrace^{d}} (-1)^{(\sum_{i=1}^{d} \mu_i)} \dim\left(E^\sigma(\lambda_1-\mu_1,\ldots,\lambda_d-\mu_d) \right).
\end{equation}
We come back to torsion-free sheaves over the projective space. We then introduce for each $\sigma\in\Sigma$ :
\begin{equation}
 \label{eq:usigma}
u_\sigma:=\sum_{\rho\in\sigma(1)} u_\rho
\end{equation}
and recall that $H$ stands for the class of a hyperplane section.
\begin{proposition}
 \label{prop:formula total chern class general}
 We have
 \begin{equation}
  \label{eq:total chern class general}
 \uc(\cE)=\prod_{\sigma\in\Sigma}\prod_{m\in M/(\sigma^\perp\cap M)} \left(1-\langle u_\sigma, m \rangle H\right)^{(-1)^{\codim(\sigma)}\dim\left( [E^\sigma](m)\right) }.
 \end{equation}
\end{proposition}
The proof follows directly from Klyachko's work \cite{Kly90}.
\begin{proof}
 From \cite[Theorem 3.2.1]{Kly90}, see also \cite[Section 4.3.1]{KnuSha98}, we have the formula
 $$
 \uc(\cE)=\prod_{\sigma\in\Sigma}\prod_{m\in M/(\sigma^\perp\cap M)} \left(1-\sum_{\rho\in\sigma(1)}\langle u_\rho, m \rangle D_\rho\right)^{(-1)^{\codim(\sigma)}\dim\left( [E^\sigma](m)\right) }
 $$
 for every locally-free equivariant sheaf over the projective space (note that we use increasing filtrations rather than decreasing filtrations as in \cite{Kly90,KnuSha98}, which accounts for the difference in signs in our formulae). Arguing as in \cite[Proposition 3.16]{Koo11}, using the existence of finite equivariant resolutions by locally-free sheaves for any equivariant torsion-free sheaf, we conclude that this formula holds for any equivariant torsion-free sheaf. The result follows.
\end{proof}
While Formula \eqref{eq:total chern class general} may be difficult to handle, using multiplicativity of the total Chern class, we can obtain simpler formulae for the quotients of the total Chern classes of two equivariant sheaves coming from elementary injections. Hence, starting from a sheaf whose total Chern class is known - typically a reflexive one - we may be able to compute the total Chern class of its subsheaves. 
\begin{proposition}
 \label{prop:ratio elementary injection}
 Consider an elementary injection $\cE\to \cF$, with parameters $(k_0,\sigma_0,m_0)$. We have the following formula in the ring $\Z[H]/\langle H^{n+1}\rangle$ :
 \begin{equation}
  \label{eq:ratio chern elementary}
  \frac{\uc(\cF)}{\uc(\cE)}=\prod_{\sigma_0\leq\sigma}\prod_{m\in W_\sigma}\left(\prod_{i=0}^{\dim(\sigma)}(1-(\langle u_\sigma,m\rangle+i)H)^{(-1)^{(\codim(\sigma)+i)}\binom{\dim(\sigma)}{i}}\right).
 \end{equation}
\end{proposition}
Recall that $W_\sigma$ was defined in \eqref{eq:definition Wsigma}.
\begin{proof}
From Lemma \ref{lem:dim diff elementary injection} and Lemma \ref{lem:precise when dim diff}, for any cone $\sigma\in\Sigma$ that doesn't contain $\sigma_0$, $E^\sigma_m=F^\sigma_m$ for all weights $m$. Hence in that case $\dim\left( [F^\sigma](m)\right)= \dim\left( [E^\sigma](m)\right)$ for any weight $m$. Together with Equation \eqref{eq:total chern class general},  we find
 $$
 \frac{\uc(\cF)}{\uc(\cE)}=\prod_{\sigma_0\leq\sigma}\prod_{m\in M/(\sigma^\perp\cap M)} \left(1-\langle u_\sigma, m \rangle H\right)^{(-1)^{\codim(\sigma)}\left(\dim\left( [F^\sigma](m)\right)- \dim\left( [E^\sigma](m)\right)\right)}.
 $$
 Using now Formula \eqref{eq:dimension Esigma crochet}, we have that for each $\sigma_0\leq\sigma$, the term
 $$
 \displaystyle\prod_{m\in M/(\sigma^\perp\cap M)} \left(1-\langle u_\sigma, m \rangle H\right)^{(-1)^{\codim(\sigma)}\left(\dim\left( [F^\sigma](m)\right)- \dim\left( [E^\sigma](m)\right)\right)}
 $$
 equals
 \begin{equation*}
 \label{eq:change of index}
  \displaystyle\prod_{m\in M/(\sigma^\perp\cap M)}\prod_{\mu\in\lbrace 0,1\rbrace^{\dim\sigma}} \left(1-\langle u_\sigma, m \rangle H\right)^{(-1)^{\codim(\sigma)+\mu}\left(\dim\left( F^\sigma(m-\mu)\right)- \dim\left( E^\sigma(m-\mu)\right)\right)},
 \end{equation*}
 where we used the notations
 $$
 (-1)^\mu:=(-1)^{\mu_1+\ldots+\mu_{\dim\sigma}}
 $$
 and
 $$
 F^\sigma(m-\mu):=F^\sigma(m-(\mu_1e_1^*+\ldots+\mu_de_d^*)),
 $$
 and similarly for $E^\sigma(m-\mu)$, identifying again
 $M/(\sigma^\perp\cap M)$ with $\Z\cdot e_1^*\oplus\ldots\oplus \Z\cdot e_d^*$. Formula \eqref{eq:ratio chern elementary} then follows from Lemma \ref{lem:dim diff elementary injection}, the definition of $W_\sigma$, and a change of index $(m,\mu)\mapsto (m+\mu,\mu)$.
\end{proof}
For saturated elementary injections, we will obtain a simpler formula, in several steps. We will need first the following lemma :
\begin{lemma}
 \label{lem:inductive formula compute ratio cherns}
 Let $(a,\um,d)\in \Z\times \Z\times \N$ with $2\leq d\leq n$. Then the following holds in $\Z[H]/\langle H^{n+1}\rangle$ :
 \begin{equation}
  \label{eq:infinite products}
  \prod_{m\geq a} \prod_{i=0}^d \left(1-(\um+m+i)H\right)^{(-1)^{n-d+i}\binom{d}{i}}=\prod_{i=0}^{d-1}\left(1-(\um+a+i)H\right)^{(-1)^{n-d+i}\binom{d-1}{i}}.
 \end{equation}
\end{lemma}
\begin{proof}
The left hand side of \eqref{eq:infinite products} can be written
$$
\prod_{j\in\N}(1-(\um+a+j)H)^{\alpha_j},
$$
where, for $j<d$,
$$
\alpha_j=\sum_{i=0}^j(-1)^{n-d+i}\binom{d}{i}
$$
and for $j\geq d$,
$$
\alpha_j=\sum_{i=0}^d(-1)^{n-d+i}\binom{d}{i}=(-1)^{n-d}\sum_{i=0}^d(-1)^{i}\binom{d}{i}=0.
$$
For $j<d$, we compute
\begin{equation*}
 \begin{array}{ccc}
 \displaystyle\sum_{i=0}^j(-1)^{i}\binom{d}{i} &=& \displaystyle \sum_{i=1}^j(-1)^{i}\binom{d-1}{i-1}+\sum_{i=1}^j(-1)^{i}\binom{d-1}{i}+1\\
                                &=&\displaystyle-\sum_{i=0}^{j-1}(-1)^{i}\binom{d-1}{i}+\sum_{i=1}^j(-1)^{i}\binom{d-1}{i}+1\\
                                &=&\displaystyle-1+(-1)^{j}\binom{d-1}{j}+1=(-1)^{j}\binom{d-1}{j},
 \end{array}
\end{equation*}
and the result follows.
\end{proof}

\begin{corollary}
 \label{prop:chern for very elementary injection}
  Let $\cE\to \cF$ be a saturated elementary injection associated to $(E^\sigma_m)\subset (F^\sigma_m)$, with parameters $(k_0,\sigma_0,m_0)$. Then
 \begin{equation}
  \label{eq:ratio chern very elementary}
  \frac{\uc(\cF)}{\uc(\cE)}=\prod_{\sigma_0\leq\sigma}\left(\prod_{i=0}^{k_0}(1-(\um_\sigma+i)H)^{(-1)^{(\codim(\sigma)+i)}\binom{k_0}{i}}\right),
 \end{equation}
 where $\um_\sigma:=\langle m_\sigma, u_\sigma\rangle$ is defined by \eqref{eq:msigma} and \eqref{eq:usigma}.
\end{corollary}

\begin{proof}
 Starting from Formula \eqref{eq:ratio chern elementary}, we need to compute for each $\sigma_0\leq\sigma$ :
 \begin{equation}
  \label{eq:Wsigma product}
 \prod_{m\in W_\sigma}\left(\prod_{i=0}^{\dim(\sigma)}(1-(\langle u_\sigma,m\rangle+i)H)^{(-1)^{(\codim(\sigma)+i)}\binom{\dim(\sigma)}{i}}\right).
 \end{equation}
 Using the fact that $W_\sigma= (m_0+\sigma_0^\perp)\cap (m_\sigma+\sigma^\vee)$, and the previous notations (i.e. the basis $(e_i)_{1\leq i\leq d}$, $m_\sigma=m_0+a_{k_0+1}e_{k_0+1}^*+\ldots+a_d e_d^*$, etc) we have the following description :
 $$
 W_\sigma=\lbrace m_0+m_{k_0+1}e_{k_0+1}^*+\ldots+m_d e_d^*\:\vert\:\forall\; k_0<j\leq d\;,\: m_j\geq a_j \rbrace.
 $$
 Hence the product in \eqref{eq:Wsigma product} equals
 $$
 \prod_{m_d\geq a_d}\ldots\prod_{m_{k_0+1}\geq a_{k_0+1}}\left(\prod_{i=0}^{d}(1-(\um_{\sigma_0}+m_{k_0+1}+\ldots+m_d+i)H)^{(-1)^{(n-d+i)}\binom{d}{i}}\right).
 $$
 The result then follows by iterating Lemma \ref{lem:inductive formula compute ratio cherns}.
\end{proof}

We can actually simplify further formula \eqref{eq:ratio chern very elementary}. For convenience, we will use sums instead of products in the ring $\Z[H]/\langle H^{n+1} \rangle$, so we introduce here a formal logarithm. 
\begin{definition}                                                                                                                                                                                                                                                          \label{eq:formal log}
Let $P(H)\in  \Z[H]/\langle H^{n+1} \rangle$ be an invertible element, written $P(H)=1+R(H)$ where $R(H)$ is a multiple of $H$. Then we set
$$
\log(P(H)):=-\sum_{i=1}^n (-1)^i \frac{(R(H))^i}{i}\in \Q[H]/\langle H^{n+1} \rangle.
$$
\end{definition}
It is a straightforward exercise to check the following properties:
\begin{lemma}
 \label{lem:log properties}
 Let $P_1,P_2\in\Z[H]/\langle H^{n+1} \rangle$ be two invertible elements. Then:
 \begin{enumerate}
  \item $\log(P_1)=\log(P_2)$ if and only if $P_1=P_2$,
  \item $\log(P_1P_2)=\log(P_1)+\log(P_2)$,
  \item $\log(P_1^{-1})=-\log(P_1)$.
 \end{enumerate}
\end{lemma}
We will derive an expansion in $H$ for Formula \eqref{eq:ratio chern very elementary}. For this we need two lemmas, and some notations.
 \begin{lemma}
  \label{lem:binomial expressions}
  Let $(p,k)\in\N^2$, and consider the expression
  $$
  A_{p,k}:=\sum_{i=0}^k\binom{k}{i}(-1)^ii^p\in\Z.
  $$
  Then :
  \begin{enumerate}
   \item if $p<k$, then $A_{p,k}=0$,
   \item if $p=k$, then $A_{p,p}=(-1)^pp!$,
   \item if $p \geq k$, then $\operatorname{sgn}(A_{p,k}) = (-1)^k$,
   \item $ A_{0,0}=1$,  $A_{0,k}=0$ for $k\geq 1$, and then for $p,k\geq 1$,
   $$
   A_{p,k} = k \Big( A_{p-1,k} - A_{p-1,k-1} \Big).$$ 
  \end{enumerate}
 \end{lemma}
  Note that
 $$
A_{p,k} = (-1)^k \, k! \, S(p,k),
 $$
 where $S(p,k)$ stands for Stirling numbers of the second kind (\cite[Chapter 13, Equation (13.13)]{VanLintWilson}). The results of Lemma \ref{lem:binomial expressions} are then classical. Note that by point $(4)$, one can compute inductively the values of $A_{p,k}$. We list here some examples :
%

\begin{equation}
 \label{eq:some values of Apk}
A_{p,k}=\begin{array}{c|rrrrrr}
p\backslash k & 0 & 1 & 2 & 3 & 4 & 5\\\hline
0 & 1 & 0 & 0 & 0 & 0 & 0\\
1 & 0 & -1 & 0 & 0 & 0 & 0\\
2 & 0 & -1 & 2 & 0 & 0 & 0\\
3 & 0 & -1 & 6 & -6 & 0 & 0\\
4 & 0 & -1 & 14 & -36 & 24 & 0\\
5 & 0 & -1 & 30 & -150 & 240 & -120
\end{array}
\end{equation}

%

We now introduce some notations. Let $(E^\sigma_m)\subset(F^\sigma_m)$ be an elementary injection with parameters $(k_0,\sigma_0,m_0)$. For any $\rho\in\sigma_0(1)$, we set 
 $$
 m_\rho:=\langle m_0, u_\rho \rangle,
 $$
 while for $\rho\in\Sigma(1)\setminus \sigma_0(1)$, if $k_0=\dim(\sigma_0)\leq n-1$, we consider the cone $\tau=\sigma_0+\rho$, and we set
 \begin{equation}
  \label{eq:definition m rho}
 m_\rho:=\um_\tau-\um_{\sigma_0}=\langle m_\tau,u_\tau\rangle - \langle m_0,u_{\sigma_0}\rangle.
 \end{equation}
 Finally, we introduce
 \begin{equation}
  \label{eq:umbigsigma}
 \um_\Sigma:=\sum_{\rho\in\Sigma(1)} m_\rho.
 \end{equation}
 \begin{lemma}
  \label{lem:sum minus one on cones}
  With the previous notations we have, for $p\leq \codim(\sigma_0)$,
\begin{equation}
 \label{eq:summinusonemsigma}
\sum_{\sigma_0\leq\sigma} (-1)^{\codim(\sigma)}\um_\sigma^p=\um_\Sigma^p.
\end{equation}
 \end{lemma}
 \begin{proof}
 The identity will actually be proved for any set of integers $(m_\rho)_{\rho\in\Sigma(1)}$, setting 
 $\um_\sigma=\sum_{\rho\in\sigma(1)}m_\rho$ for any $\sigma\in\Sigma$, and $\um_\Sigma=\sum_{\rho\in\Sigma(1)} m_\rho$ (forgetting that the $m_\rho$'s come from an elementary injection). 
  We first deal with the case $p=0$. In that case, the proof follows from the following :
  $$
  \sum_{\sigma_0\leq\sigma} (-1)^{\codim(\sigma)}=\sum_{i=k_0}^n\binom{n+1-k_0}{i-k_0}(-1)^{n-i}=1
  $$
  that can be obtained by noting that there are $\binom{n+1-k_0}{i-k_0}$ cones of dimension $i$ that contain $\sigma_0$. 
 We then proceed to a proof by induction on $p$. Assume first that there is $\rho\in\sigma_0(1)$ with $m_\rho\neq 0$. Then consider
 $$
 \begin{array}{ccc}
 \displaystyle \sum_{\sigma_0\leq\sigma} (-1)^{\codim(\sigma)}\um_\sigma^p & = &\displaystyle\sum_{\sigma_0\leq\sigma} (-1)^{\codim(\sigma)}\Big(\sum_{\rho\in\sigma(1)}m_\rho\Big)^p\\
    & = &\displaystyle\sum_{\sigma_0\leq\sigma} (-1)^{\codim(\sigma)}\Big(\um_{\sigma_0}+\sum_{\rho\in\sigma(1)\setminus\sigma_0(1)}m_\rho\Big)^p\\
     & = &\displaystyle\sum_{\sigma_0\leq\sigma} (-1)^{\codim(\sigma)} \Big(\sum_{q=0}^p \binom{p}{q} \um_{\sigma_0}^{p-q}\Big(\sum_{\rho\in\sigma(1)\setminus\sigma_0(1)}m_\rho\Big)^q \Big)\\
        & = &\displaystyle \sum_{q=0}^p \binom{p}{q}\um_{\sigma_0}^{p-q}\Big(\sum_{\sigma_0\leq\sigma} (-1)^{\codim(\sigma)} \Big(\sum_{\rho\in\sigma(1)\setminus\sigma_0(1)}m_\rho\Big)^q \Big).
 \end{array}
 $$
 Now, for $q<p$, we can use the induction hypothesis on the set of integers $(m_\rho')_{\rho\in\Sigma(1)}$ defined by $m_\rho'=m_\rho$ if $\rho\notin\sigma_0(1)$ and $m_\rho'=0$ if $\rho\in\sigma_0(1)$. This gives
 $$
 \begin{array}{ccc}
 \sum_{\sigma_0\leq\sigma} (-1)^{\codim(\sigma)}\um_\sigma^p & = & \displaystyle \sum_{q=0}^{p-1} \binom{p}{q}\um_{\sigma_0}^{p-q}\Big(\um_\Sigma-\um_{\sigma_0}\Big)^q\\
 & & \displaystyle+\sum_{\sigma_0\leq\sigma} (-1)^{\codim(\sigma)} \Big(\sum_{\rho\in\sigma(1)\setminus\sigma_0(1)}m_\rho\Big)^p .
 \end{array}
 $$
By induction on the number of $\rho$'s such that $m_\rho\neq 0$, we also obtain 
$$
\sum_{\sigma_0\leq\sigma} (-1)^{\codim(\sigma)} \Big(\sum_{\rho\in\sigma(1)\setminus\sigma_0(1)}m_\rho\Big)^p= (\um_\Sigma-\um_{\sigma_0})^p,
$$
and we can conclude in that case that 
$$
\sum_{\sigma_0\leq\sigma} (-1)^{\codim(\sigma)}\um_\sigma^p=\displaystyle \sum_{q=0}^{p} \binom{p}{q}\um_{\sigma_0}^{p-q}(\um_\Sigma-\um_{\sigma_0})^q=(\um_{\sigma_0}+\um_\Sigma-\um_{\sigma_0})^p=\um_\Sigma^p.
$$
What remains is to initialize the last induction used, that is to deal with the case where for each $\rho\in\sigma_0(1)$, $m_\rho=0$. So we are left with showing :
$$
\sum_{k=\dim(\sigma_0)}^n \left(\sum_{\sigma_0\leq\sigma,\,\sigma\in\Sigma(k)} (-1)^{n-k}\Big(\sum_{\rho\in\sigma(1)\setminus\sigma_0(1)}m_\rho\Big)^p\,\right)=\Big(\sum_{\rho\in\Sigma(1)\setminus \sigma_0(1)}m_\rho\Big)^p.
$$
Re-indexing, setting $\codim(\sigma_0)=r$, it is enough to show that for any $(r+1)$-tuple $(m_j)_{0\leq j\leq r}\in\Z^{r+1}$ (that corresponds to $(m_\rho)_{\rho\notin \sigma_0(1)}$) and any $1\leq p \leq r$, we have
\begin{equation}
 \label{eq:alternating sum mis}
\sum_{i=1}^{r} (-1)^i \sum_{\substack{J \subset \{0,\dots,r\} \\ |J|=i}} \Big(\sum_{j \in J} m_j\Big)^p=(-1)^r \Big(\sum_{j=0}^{r} m_j \Big)^p.
\end{equation}
To show this, consider the generating function
$$
F(y) = \sum_{J \subset \{0,\dots,r\}} (-1)^{|J|} e^{y \sum_{j \in J} m_j} = \prod_{j=0}^{r} (1 - e^{y m_j}).
$$
Formally, we can develop in power series :
$$
e^{y \sum_{j \in J} m_j} = \sum_{q=0}^{\infty} \frac{y^q}{q!} \Big(\sum_{j \in J} m_j\Big)^q,
$$
and compute the coefficient in front of $\displaystyle\frac{y^p}{p!}$ in $F(y)$ :
$$
\sum_{J \subset \{0,\dots,r\}} (-1)^{|J|} \Big(\sum_{j \in J} m_j\Big)^p = \sum_{i=0}^{r+1} (-1)^i \sum_{\substack{J \subset \{0,\dots,r\} \\ |J|=i}} \Big(\sum_{j \in J} m_j\Big)^p.
$$
On the other hand, $y$ divides each factor $(1-e^{y m_j})$, hence $F(y)=\prod_{j=0}^{r} (1 - e^{y m_j})$ starts in degree $y^{r+1}$.  
Thus, for $1 \leq p \leq r$, the coefficient in front of $y^p$ in $F(y)$ vanishes, which gives \eqref{eq:alternating sum mis}.
 \end{proof}
We can now state the formula that will be useful in the prescription problem for Chern classes :
\begin{proposition}
 \label{prop:log expansion chern saturated elementary}
  Let $\cE\to \cF$ be a saturated elementary injection with parameters $(k_0,\sigma_0,m_0)$, with $k_0\geq 1$. Then we have an expansion in $\Z[H]/\langle H^{n+1} \rangle$ :
 \begin{equation}
  \label{eq:log expansion saturated elementary}
  \log\left(\frac{\uc(\cF)}{\uc(\cE)}\right)=-\sum_{k=k_0}^n \Big(\sum_{l=k_0}^k \binom{k}{l}\, A_{l,k_0}\, \um_\Sigma^{k-l} \Big)\, \frac{H^k}{k},
 \end{equation}
 where recall $\um_\Sigma$ is defined by \eqref{eq:umbigsigma}.
 \end{proposition}
 
 \begin{proof}
  By \eqref{eq:ratio chern very elementary}, we need to compute
  $$
  \sum_{\sigma_0\leq\sigma}\sum_{i=0}^{k_0}(-1)^{(\codim(\sigma)+i)}\binom{k_0}{i}\log\left(1-(\um_\sigma+i)H\right).
  $$
  Using the definition of $\log$, this gives
  $$
  -\sum_{\sigma_0\leq\sigma}(-1)^{\codim(\sigma)}\sum_{i=0}^{k_0}(-1)^{i}\binom{k_0}{i}\sum_{j=1}^n  \frac{(\um_\sigma+i)^j}{j}H^j.
  $$
  The result is then straightforward, using Lemmas \ref{lem:binomial expressions} and \ref{lem:sum minus one on cones}.
 \end{proof}
 We also have, the following simple formula for ratios of Chern classes under saturated elementary injections :
 \begin{corollary}
  \label{cor:simplest product formula Chern ratios saturated elementary}
  Let $\cE\to \cF$ be a saturated elementary injection with parameters $(k_0,\sigma_0,m_0)$, with $k_0\geq 1$. Then the following holds in $\Z[H]/\langle H^{n+1} \rangle$ :
  \begin{equation}
   \label{eq:simplest product formula Chern ratios}
  \frac{\uc(\cF)}{\uc(\cE)}=\prod_{i=0}^{k_0}(1-(\um_\Sigma+i)H)^{(-1)^i\binom{k_0}{i}},
  \end{equation}
  where recall $\um_\Sigma$ is defined by \eqref{eq:umbigsigma}.
 \end{corollary}
\begin{proof}
From Lemma \ref{lem:binomial expressions} point $(1)$, $A_{p,k}=0$ for $p < k$. Then by Proposition \ref{prop:log expansion chern saturated elementary}, we have 
$$
\log\left(\frac{\uc(\cF)}{\uc(\cE)}\right)=-\sum_{k=1}^n \Big(\sum_{l=0}^k \binom{k}{l}\, A_{l,k_0}\, \um_\Sigma^{k-l} \Big)\, \frac{H^k}{k}.
$$
Unravelling the identities used in the proof of Proposition \ref{prop:log expansion chern saturated elementary}, we have
$$
\sum_{l=0}^k \binom{k}{l}\, A_{l,k_0}\, \um_\Sigma^{k-l} = \sum_{i=0}^{k_0}\binom{k_0}{i} (-1)^i(\um_\Sigma+i)^k,
$$
and the result follows from the definition of $\log$ and its properties from Lemma \ref{lem:log properties}.
\end{proof}

\begin{remark}
 \label{rem:interpretation of weights}
 Using the arguments of the proof of Proposition \ref{prop:elementary injection geometric}, together with Kool's description of pure equivariant sheaves (cf. \cite[Proposition 2.8 and Theorem 2.10]{Koo11}), one can give a geometric interpretation of the weight $\um_\Sigma$. If $\cQ=\cF/\cE$ denotes the quotient of an elementary injection with parameters $(\sigma_0,m_0)$, $V=V(\sigma_0)=\mathrm{Supp}(\cQ)$, and $\Sigma_V$ stands for the fan of $V$, there is a bijection between the rays in $\Sigma_V$ and the rays in $\Sigma(1)\setminus\sigma(1)$ (\cite[Section 3.2]{CLS}, here we use that we work on the projective space). Then, using the family of multifiltrations for $\cQ$ (cf proof of Proposition \ref{prop:elementary injection geometric}), setting $\iota : V \to \C\P^n$ the inclusion, we obtain :
 $$
 \cQ=\Big(\iota_*\cO_V(-\sum_{\rho\in\Sigma_V(1)}m_\rho \check D_\rho)\Big)\otimes \cO_{\P^n}(-\sum_{\rho\in\sigma_0(1)} m_\rho D_\rho),
 $$
 where $\check D_\rho$ stand for the invariant divisors on the projective space $V$. From this we deduce that
 $$
 \cQ\simeq \iota_*\Big(\cO_{V}(-\um_\Sigma)\Big)
 $$
 where the isomorphism is at the level of sheaves, but not of torus-equivariant sheaves. To recover the equivariant structure on $\cQ$, one needs the data of $m_0\in M$ and all the $m_\rho$'s, for $\rho\in \Sigma(1)\setminus\sigma_0(1)$.
\end{remark}
We end this section by noting that Formula \eqref{eq:log expansion saturated elementary} is also valid for a non saturated elementary injection at order $k_0+1$ :
\begin{proposition}
 \label{prop:log expansion chern non saturated}
  Let $\cE\to \cF$ be an elementary injection with parameters $(\sigma_0,m_0)$ and with $k_0=\dim(\sigma_0)\geq 1$. Then we have an expansion in $\Z[H]/\langle H^{n+1} \rangle$ :
 \begin{equation*}
  \label{eq:log expansion non saturated}
  \log\left(\frac{\uc(\cF)}{\uc(\cE)}\right)=(-1)^{k_0-1}(k_0-1)! \, H^{k_0}-\Big(A_{k_0,k_0}\, \um_\Sigma + \frac{A_{k_0+1,k_0}}{k_0+1} \Big)H^{k_0+1} + O(H^{k_0+2})
 \end{equation*}
 where recall $\um_\Sigma$ is defined by \eqref{eq:umbigsigma}, and where $O(H^{k_0+2})$ stands for a multiple of $H^{k_0+2}$ in the ring $\Z[H]/\langle H^{n+1} \rangle$.
 \end{proposition}
\begin{proof}
We prove this by decreasing induction on $k_0$. Note that by definition, any $n$-elementary or $(n-1)$-elementary injection is automatically saturated, hence those cases follow from Proposition \ref{prop:log expansion chern saturated elementary}. Assume then $k_0\leq n-2$. Consider the smallest equivariant torsion-free sheaf $\cG$ such that $\cE\subset\cF\subset \cG$ and $\cE\subset \cG$ is a saturated elementary injection with parameters $(\sigma_0,m_0)$. It may be obtained by considering the family of multifiltrations $(G^\sigma_m)$ defined as follows, using notations from Section \ref{sec:elementary injections} (in particular Equation \eqref{eq:msigma}). Let $(E^\sigma_m)\subset (F^\sigma_m)$ stand for the families of multifiltrations of $\cE$ and $\cF$. For $\sigma_0\nleq \sigma$, $G^\sigma_m=E^\sigma_m=F^\sigma_m$ for any $m\in M$. For $\sigma_0\leq \sigma$, if $m\notin (m_0+\sigma_0^\perp)\cap(m_\sigma+\sigma^\vee)$, then $G^\sigma_m=E^\sigma_m=F^\sigma_m$, while if $m\in (m_0+\sigma_0^\perp)\cap(m_\sigma+\sigma^\vee)$, then $G^\sigma_m=F^\sigma_m\oplus L$, where recall the line $L$ is defined by $F^{\sigma_0}_{m_0}=E^{\sigma_0}_{m_0}\oplus L$. Then, the associated sheaf $\cG$ provides the desired saturated elementary injection, with parameters $(\sigma_0,m_0)$. By construction of the sets $W_\sigma$, for $\sigma\in\Sigma(k_0)\cup\Sigma(k_0+1)$, $F^\sigma_m$ agrees with $G^\sigma_m$. Hence, from the proof of Theorem \ref{theo:existence decomposition elementary injections}, one may factorize $\cF\subset \cG$ into elementary injections that are $k$-elementary, with $k\geq k_0+2$. We can apply the induction hypothesis to those elementary injections, and thus conclude that
$$
\log\left(\frac{\uc(\cG)}{\uc(\cF)}\right)=O(H^{k_0+2}).
$$
 Then, 
 \begin{eqnarray*}
 \log\left(\frac{\uc(\cG)}{\uc(\cE)}\right) &=\displaystyle\log\left(\frac{\uc(\cG)}{\uc(\cF)}\right)+\log\left(\frac{\uc(\cF)}{\uc(\cE)}\right)\\
  & = \displaystyle\log\left(\frac{\uc(\cF)}{\uc(\cE)}\right)+O(H^{k_0+2}).
 \end{eqnarray*}
 As $\cE\subset\cG$ is saturated, we can use formula \eqref{eq:log expansion saturated elementary} for $\log\left(\frac{\uc(\cG)}{\uc(\cE)}\right)$, and as $(F^\sigma_m)$ agrees with $(G^\sigma_m)$ as long as $\dim(\sigma)\leq k_0+1$, we deduce that the weight $\um_\Sigma$ corresponding to $(E^\sigma_m)\subset (F^\sigma_m)$ is the same as the one for $(E^\sigma_m)\subset(G^\sigma_m)$. This concludes the proof.
\end{proof}

\section{Applications}
\label{sec:applications}

We come back to our initial problem, namely producing rank $2$ torsion-free sheaves with prescribed Chern classes over the projective space. We would like in particular those Chern classes to satisfy the constraints of rank $2$ vector bundle Chern classes, that we now recall. First, if $\cE$ is a rank $2$ vector bundle over the projective space $\C\P^n$, we have the vanishing of $\uc_i(\cE)$ for $i\geq 3$. Then the Schwarzenberger's constraints (\cite[Appendix 1]{Hirz}) for $\uc(\cE):=1+\uc_1\cdot H+\uc_2\cdot H^2$ can be formulated as follows (cf \cite[Section 2]{HoSch}). For all $m\in\N$ such that $2\leq m\leq n$, we must have
\begin{equation}
 \label{eq:Schwarzenberger conditions}
  \sum_{j=2}^m\sum_{i=1}^{[\frac{j}{2}]} (-1)^i e_{m,j}\left( \binom{j-i}{i}+ \binom{j-i-1}{i-1}\right)\uc_1^{j-2i}\uc_2^i \equiv 0 \: \mathrm{ mod }\: m!,
\end{equation}
where the $e_{m,j}\in\N$ are defined by
$$
\sum_{j=1}^m e_{m,j}t^j=t(t+1)\ldots(t+m-1).
$$

\subsection{Obstructions to smoothability}
\label{sec:negative results}
We begin with some negative results. We shall prove here Theorem \ref{theo:intro obstructions smoothability} and Corollary \ref{cor:intro obstructions smoothability}. Let then $\cE$ be a  torus-equivariant rank $2$ torsion-free sheaf with normalised reflexive hull $\cF=(\cE^\vee)^\vee$ on the projective space (we use the normalisation as in Corollary \ref{cor:chern classes normalized}). Let $\cQ=\cF/\cE$, and set $q=\codim(\cQ)=n-\dim(\cQ)$. From the proof of Theorem \ref{theo:existence decomposition elementary injections}, $\cE\subset\cF$ can be factorized by a sequence of $k$-elementary injections, with $k\geq q$. Denote by 
$$
\cE:=\cE_{n+1}\subset\cE_n\subset  \ldots\subset \cE_q=\cF
$$
the subsequence of injections obtained by setting, for each $k$, the inclusion $\cE_{k+1}\subset  \cE_k$ to be the composition of all the successive $k$-elementary injections appearing in the factorization of $\cE\subset\cF$. Then from Proposition \ref{prop:elementary injection geometric}, $\cE_k/\cE_{k+1}$ is pure of dimension $n-k$.
Hence,
$$
0=\cE/\cE_{n+1}\subset \cE_n/\cE_{n+1}\subset\ldots\subset \cE_q/\cE_{n+1}=\cQ
$$
is the torsion filtration of $\cQ$, and for each $k$,
$$
Q_{n-k}(\cQ)=\cE_k/\cE_{k+1},
$$
so that the vanishing of $Q_{n-k}(\cQ)$ means that no $k$-elementary injection appears in the factorization of $\cE\subset\cF$. Moreover, for each $k$, if $Q_{n-k}(\cQ)=0$, 
\begin{equation}
 \label{eq:log E_k E_k+1 vanish}
\log\Big(\frac{\uc(\cE_k)}{\uc(\cE_{k+1})} \Big)=0,
\end{equation}
while in general, by Proposition \ref{prop:log expansion chern non saturated}, 
\begin{equation}
 \label{eq:log E_k E_k+1}
\log\Big(\frac{\uc(\cE_k)}{\uc(\cE_{k+1})} \Big)=(-1)^{k-1}(k-1)!p_k\, H^k+O(H^{k+1}),
\end{equation}
where $p_k$ stands for the number of $k$-elementary injections appearing in the factorization of $\cE\subset\cF$. Finally,
\begin{equation}
 \label{eq:log F E obstruction result}
\log\Big(\frac{\uc(\cF)}{\uc(\cE)} \Big)=\sum_{k=q}^n\log\Big(\frac{\uc(\cE_{k})}{\uc(\cE_{k+1})} \Big).
\end{equation}
With this at hand, we can prove Theorem \ref{theo:intro obstructions smoothability}.

\subsubsection{Case $q\geq 4$.} From \eqref{eq:log F E obstruction result} and \eqref{eq:log E_k E_k+1},
$$
\log\Big(\frac{\uc(\cF)}{\uc(\cE)} \Big)=\sum_{k=q}^n\log\Big(\frac{\uc(\cE_{k})}{\uc(\cE_{k+1})} \Big)=(-1)^{q-1}(q-1)!p_q\, H^q+O(H^{q+1})
$$
with $p_q\neq 0$ by assumption on $\codim(\cQ)$. If in addition $\uc_3(\cF)=0$, then by Corollary \ref{cor:loc free iff c3 vanishes}, $\cF$ is locally free and then $\uc_i(\cF)=0$ for $i\geq 3$. Hence we can write 
$$\uc(\cF)=1+\uc_1 \,H+\uc_2\, H^2.$$
Then we have
$$
\displaystyle\log\Big(\frac{1+\uc_1 \,H+\uc_2\, H^2}{1+\uc_1 H+\uc_2(\cE) H^2+\ldots +\uc_n(\cE) H^n}\Big)=\log\Big(\frac{\uc(\cF)}{\uc(\cE)} \Big)  =\alpha_q \, H^q+O(H^{q+1})
$$
for some $\alpha_q\in\Z\setminus\lbrace 0\rbrace$.
By the expansion of $\log$ in Definition \ref{eq:formal log}, this implies that $\uc_i(\cE)=0$ for $3\leq i \leq q-1$ and $\uc_q(\cE)\neq 0$.

Assuming now $\uc_3(\cF)\neq 0$, we obtain
\begin{eqnarray*}
\displaystyle\log\Big(\frac{1+\uc_1(\cF) \,H+\uc_2(\cF)\, H^2+\ldots+\uc_n(\cF)H^n}{1+\uc_1(\cE) H+\uc_2(\cE) H^2+\ldots +\uc_n(\cE) H^n}\Big)=\alpha_q \, H^q+O(H^{q+1})
\end{eqnarray*}
which this time implies that $\uc_i(\cE)=\uc_i(\cF)$ for $i\leq q-1$. In particular $\uc_3(\cE)=\uc_3(\cF)$ is non zero. This achieves the proof of Theorem \ref{theo:intro obstructions smoothability}, item $(1)$.

\begin{remark}
 Note that the previous proof shows that if $\cF=(\cE^\vee)^\vee$ is locally free and if $q=\codim(\cF/\cE)\geq 3$, then $\uc_q\neq 0$ and $\cE$ is not smoothable.
\end{remark}

\subsubsection{Case $q=2$.} We move on to the proof of item $(2)$, and assume that $\cE$ is semistable and $Q_{n-3}(\cQ)=0$. The latter implies,  by \eqref{eq:log F E obstruction result}, \eqref{eq:log E_k E_k+1} and \eqref{eq:log E_k E_k+1 vanish}, that
$$
\log\Big(\frac{\uc(\cF)}{\uc(\cE)} \Big)=\log\Big(\frac{\uc(\cE_{2})}{\uc(\cE_{3})} \Big)+O(H^4),
$$
with $\cE_3\subsetneq \cE_2$ as $q=2$. We now use Proposition \ref{prop:log expansion chern non saturated}, applied to the successive $2$-elementary injections in $\cE_3\subset\cE_2$. Let $p_2>0$ stand for the number of $2$-elementary injections in the factorization of $\cE_3\subset\cE_2$, and let $\um_{\Sigma,i}$ stand for the weight of the $i$-th one, as defined in \eqref{eq:umbigsigma}. Then, we obtain
$$
\log\Big(\frac{\uc(\cE_2)}{\uc(\cE_3)} \Big)=-p_2\,H^2-2\Big(\sum_{i=1}^{p_2}\um_{\Sigma,i}+p_2\Big)\,H^3+O(H^4).
$$
We deduce that
$$
\displaystyle\log\Big(\frac{1+\uc_1(\cF) \,H+\ldots+\uc_n(\cF)H^n}{1+\uc_1(\cE) H+\ldots +\uc_n(\cE) H^n}\Big)=-p_2\,H^2-2\Big(\sum_{i=1}^{p_2}\um_{\Sigma,i}+p_2\Big)\,H^3+O(H^4).
$$
As $\cF$ is the reflexive hull of $\cE$, $c_1(\cF)=c_1(\cE)$. Then, the $\log$-expansion in $H$ yields
\begin{equation}
\label{eq:system case q is 2}
 \left\{ 
 \begin{array}{ccc}
 \uc_2(\cF)-\uc_2(\cE)  &= & -p_2 \\
\uc_3(\cF)-\uc_3(\cE)+\uc_1(\cF)(\uc_2(\cE)-\uc_2(\cF)) &=   & -2(\sum_{i=1}^{p_2}\um_{\Sigma,i}+p_2).
 \end{array}
 \right.
\end{equation}
By contradiction, let's assume that $\uc_3(\cE)=0$. Then \eqref{eq:system case q is 2} gives
$$
\uc_3(\cF)+p_2\uc_1(\cF) =   -2\Big(\sum_{i=1}^{p_2}\um_{\Sigma,i}+p_2\Big)
$$
which is equivalent to
\begin{equation}
 \label{eq:the equality q is 2 case}
 \uc_3(\cF)+\sum_{i=1}^{p_2}\big(\uc_1(\cF)+2\um_{\Sigma,i}+2\Big) =   0.
\end{equation}
We need here an estimate on the $\um_{\Sigma,i}$'s. To explain the argument, we will assume for a moment that 
$\cF$ is given by the normalised family of filtrations $(\C^2,F^\rho(\bullet))_{\rho\in\Sigma(1)}$ given  by 
 \begin{equation}
  \label{eq:starting sheaf case q is 2}
 F^\rho(i)=\left\{
 \begin{array}{ccc}
                  \lbrace 0 \rbrace & \mathrm{ if } & i< -c_\rho \\
                 L_\rho & \mathrm{ if } & -c_\rho\leq i< 0 \\
\C^2 & \mathrm{ if } & 0\leq i 
                  \end{array}
                  \right.
 \end{equation}
 where  for all $\rho\in\Sigma(1)$, $L_\rho\subset\C^2$, $c_\rho\in\N$, {\it and} with $L_\rho\neq L_{\rho'}$ for $\rho\neq\rho'$. 
 From Lemma \ref{lem:reflexive hull}, the family of multifiltrations $(F^\sigma_m)$ for $\cF$ is then given, for $m\in M$ and $\sigma\in\Sigma$, by
\begin{equation}
  \label{eq:starting sheaf multifiltration q is 2}
 F^\sigma_m=\left\{
 \begin{array}{ccc}
                 L_\rho & \mathrm{ if } & -c_\rho\leq \langle m, u_\rho \rangle < 0\:\mathrm{ and }\: \forall\,\rho'\neq \rho, \: 0\leq \langle m, u_{\rho'} \rangle\\
\C^2 & \mathrm{ if } &\forall\, \rho\in\sigma(1),\: 0\leq  \langle m, u_\rho \rangle \\

                  \lbrace 0 \rbrace & \mathrm{ otherwise } & \\
                  \end{array}
                  \right.
 \end{equation}
where in the first case above, $\rho$ and $\rho'$ stand for rays in $\sigma(1)$. We wish to establish a lower bound for $\um_{\Sigma,1}$. The first $2$-elementary injection in the factorization of $\cE_3\subset\cE_2=\cF$ is obtained by picking a cone $\sigma_0\in\Sigma(2)$ and a weight $m_0\in M/(\sigma_0^\perp\cap M)$, and lowering the dimension of $F^{\sigma_0}_{m_0}$ by one, as described in the proof of Theorem \ref{theo:existence decomposition elementary injections}. Moreover, to make sure that we obtain a family of multifiltrations, the weight $m_0$ must be minimal for $\leq_{\sigma_0}$ with the property that $F^{\sigma_0}_{m_0}\neq\lbrace 0\rbrace$.  We are left with picking $$m_0=(-c_{\rho_j},0)\in M/(\sigma_0^\perp\cap M)\simeq \Z e_j^*\oplus\Z e_i^*,$$ with $\sigma_0=\rho_j+\rho_i\in\Sigma(2)$. By definition of $\um_\Sigma$ (Equation \eqref{eq:umbigsigma}) and by the description of $(F^\sigma_m)$ in \eqref{eq:starting sheaf multifiltration q is 2}, we deduce that for this choice, $\um_{\Sigma,1}=-c_{\rho_j}$. Hence, setting $c_{\rho_0}=\max\lbrace c_\rho\:\vert\: \rho\in\Sigma(1)\rbrace$, we deduce that
$\um_{\Sigma,1}\geq -c_{\rho_0}$.
Then, to produce the second elementary injection, we iterate this procedure. By considering the cases where the new chosen cone is either $\sigma_0$ or different from $\sigma_0$, we reach again the conclusion that $
\um_{\Sigma,2}\geq -c_{\rho_0}$. A simple induction then implies that for any $i$, $\um_{\Sigma,i}\geq -c_{\rho_0}$. Hence,
\begin{equation*}
 \uc_1(\cF)+2\um_{\Sigma,i}+2 \geq \uc_1(\cF)-2 c_{\rho_0}+2\geq 2
\end{equation*}
where we used the semistability of $\cF$ that implies, with Corollary \ref{cor:slope stable rank two}, that $c_{\rho_0}\leq \sum_{\rho\neq \rho_0} c_\rho$ and thus $2c_{\rho_0} \leq \sum_{\rho\in\Sigma(1)} c_\rho =\uc_1(\cF)$. We deduce that
$$
\uc_3(\cF)+\sum_{i=1}^{p_2}\Big(\uc_1(\cF)+2\um_{\Sigma,i}+2\Big) \geq \uc_3(\cF)+2p_2>0
$$
which contradicts \eqref{eq:the equality q is 2 case}.

This settles the case when $L_\rho\neq L_{\rho'}$ for $\rho\neq \rho'$. In the general case, the argument can be adapted as follows. Set, for any line $L\subset \C^2$, 
$$S_L:=\sum_{L_\rho=L} c_\rho$$ and
$$
S_{\max}:=\max\lbrace S_L\:\vert\: L\subset \C^2 \rbrace
$$
where the max is over all lines in $\C^2$. Using the definition of $\um_\Sigma$, one can easily adapt the argument to show that in that case, for all $i$, one has $\um_{\Sigma,i}\geq -S_{\max}$. Moreover, semistability implies that $\uc_1(\cF)\geq 2S_{\max}$. Then, the same inequalities as before, replacing $c_{\rho_0}$ by $S_{\max}$, lead to the same conclusion.

\subsubsection{Proof of Corollary \ref{cor:intro obstructions smoothability}.} This follows from Theorem \ref{theo:intro obstructions smoothability}, noting that both the hypothesis and the result of Corollary \ref{cor:intro obstructions smoothability} are invariant under tensoring $\cE$ by a line bundle. Hence, we may assume $(\cE^\vee)^\vee$ to be normalised, and Theorem \ref{theo:intro obstructions smoothability} applies. The non-vanishing of some Chern class $\uc_i(\cE)$ with $i\geq 3$ is preserved in flat families, preventing the existence of a smoothing.

\subsection{Prescribing Chern classes}
\label{sec:prescription}
We start now from a reflexive equivariant sheaf $\cE$, whose Chern polynomial is known from Sections \ref{sec:chern polynomial} and \ref{sec:normal forms}, and such that its first two Chern classes satisfy \eqref{eq:Schwarzenberger conditions}. We then produce inductively new torsion-free equivariant sheaves $\cE_i\subset\cE$ such that for each $i$, $\cE$ is the reflexive hull of $\cE_i$, and 
$$
\cE_p\to\cE_{p-1}\to\ldots \to \cE_0=\cE
$$
is a sequence of saturated elementary injections. Our goal is then to make sure that the following holds in $\Z[H]/\langle H^{n+1}\rangle$: 
$$
\frac{\uc(\cE)}{\uc(\cE_p)}=\frac{1+\uc_1 H+\ldots+ \uc_nH^n}{1+\uc_1 H+\uc_2 H^2}
$$
where the $\uc_i$'s are the Chern numbers of $\cE$ and $\uc(\cE_p)=1+\uc_1 H+\uc_2 H^2$. To achieve this program, we use
$$
\frac{\uc(\cE)}{\uc(\cE_p)}=\prod_{i=0}^{p-1} \frac{\uc(\cE_i)}{\uc(\cE_{i+1})},
$$
and the formula obtained in Proposition \ref{prop:log expansion chern saturated elementary}. Note that to make sure that $\cE=(\cE_i^\vee)^\vee$, from Lemma \ref{lem:reflexive hull}, we need each $\cE_{i+1}\to\cE_i$ to be a $k_i$-elementary injection with $k_i\geq 2$. We will actually assume $k_i \geq 3$, so that $\uc_2$ is preserved. From Lemma \ref{lem:log properties}, solving 
$$
\prod_{i=0}^{p-1} \frac{\uc(\cE_i)}{\uc(\cE_{i+1})}=\frac{1+\uc_1 H+\ldots+ \uc_nH^n}{1+\uc_1 H+\uc_2 H^2}
$$
is equivalent to solving 
\begin{equation}
 \label{eq:solving log}
\sum_{i=0}^{p-1} \log\left(\frac{\uc(\cE_i)}{\uc(\cE_{i+1})}\right)=\log\left(1+\uc_1 H+\ldots+ \uc_nH^n\right)-\log\left(1+\uc_1 H+\uc_2 H^2\right).
\end{equation}

 Our starting point will be a reflexive toric sheaf normalized by $b_\rho=0$ for all $\rho\in\Sigma(1)$ (in the notations of Section \ref{sec:normal forms}), with family of filtrations $(\C^2,E^\rho(\bullet))_{\rho\in\Sigma(1)}$ given  by
 \begin{equation}
  \label{eq:starting sheaf}
 E^\rho(i)=\left\{
 \begin{array}{ccc}
                  \lbrace 0 \rbrace & \mathrm{ if } & i< -c_\rho \\
                 L_\rho & \mathrm{ if } & -c_\rho\leq i< 0 \\
\C^2 & \mathrm{ if } & 0\leq i 
                  \end{array}
                  \right.
 \end{equation}
 where $L_\rho=\C\cdot u_\rho\subset N_\C\simeq \C^2$, and $c_\rho\in\N$ for all $\rho\in\Sigma(1)$, with at least one $c_\rho\neq 0$ (say $c_{\rho_0}$). Note that with this choice, for $\rho\neq\rho'$, $L_\rho\cap L_{\rho'}=\lbrace 0 \rbrace$. From Lemma \ref{cor:chern classes normalized}, the $k$-th Chern number of $\cE$ is given by 
 $$\uc_k=s_k=\sum_{\sigma\in\Sigma(k)} \prod_{\rho\in\sigma(1)}c_\rho.$$
 By Lemma \ref{lem:reflexive hull}, the family of multifiltrations $(E^\sigma_m)$ for $\cE$ is then given, for $m\in M$ and $\sigma\in\Sigma$, by
\begin{equation}
  \label{eq:starting sheaf multifiltration}
 E^\sigma_m=\left\{
 \begin{array}{ccc}
                 L_\rho & \mathrm{ if } & -c_\rho\leq \langle m, u_\rho \rangle < 0\:\mathrm{ and }\: \forall\,\rho'\neq \rho, \: 0\leq \langle m, u_{\rho'} \rangle\\
\C^2 & \mathrm{ if } &\forall\, \rho\in\sigma(1),\: 0\leq  \langle m, u_\rho \rangle \\

                  \lbrace 0 \rbrace & \mathrm{ otherwise } & \\
                  \end{array}
                  \right.
 \end{equation}
where in the first case above, $\rho$ and $\rho'$ stand for rays in $\sigma(1)$.

Our first task is then to produce sequences of elementary injections where the term $\um_\Sigma$ that appears in \eqref{eq:log expansion saturated elementary} is well controlled. We will prove
\begin{proposition}
 \label{prop:existence sequence saturated injections}
 For each $(p_3,\ldots,p_n)\in\N^{n-2}$, and each  $3\leq k \leq n$, 
 there exist rank $2$ torsion-free equivariant sheaves $(\cE_{k,j})_{0\leq j\leq p_k}$ together with saturated $k$-elementary injections $\cE_{k,j}\to\cE_{k,j-1}$ for $1\leq j\leq p_k$ such that :
 \begin{itemize}
  \item $\cE_{3,0}=\cE$, 
  \item for $k\geq 4$,  $\cE_{k,0}=\cE_{k-1,p_{k-1}}$, 
  \item the weight $\um_\Sigma^{k,j}$ associated to  $\cE_{k,j}\to\cE_{k,j-1}$ as in \eqref{eq:umbigsigma} is 
 \begin{equation}
  \label{eq:umsigmakj}
 \um_\Sigma^{k,j}=-c_{\rho_0}+\sum_{3\leq i < k}p_i+(j-1).
 \end{equation}

 \end{itemize}
\end{proposition}

\begin{proof}
We will produce the elementary injections inductively, starting from $k=3$. We will use the notations introduced in the previous sections for elementary injections, using sub or superscript $(k,j)$ to recall the dependence on the injection $\cE_{k,j} \to \cE_{k,j-1}$. Assume that $p_3 > 0$ (otherwise just move to the first non zero $p_j$, and set $\cE_{i,0}=\cE$ for $i\leq j$). Set $\cE_{3,0}=\cE$. Then consider a cone $\sigma_3=\rho_0+\rho_1+\rho_2\in\Sigma(3)$ containing $\rho_0$ (recall that $c_{\rho_0}\neq 0$). Fix the weight $m_{3,1}=(-c_{\rho_0},0,0)$ written in the basis $(e_0^*,e_1^*,e_2^*)$ of $M/(\sigma_3^\perp\cap M)$, with $e_i$ a primitive generator of $\rho_i$. The couple $(\sigma_3,m_{3,1})$ will be the parameters of our first elementary injection. From \eqref{eq:starting sheaf multifiltration}, $E^{\sigma_3}_{m_{3,1}}=L_{\rho_0}$, that we will denote $L$ from now to ease notations. By choice of $m_{3,1}$, we note that for any class $m<_{\sigma_3} m_{3,1}$ in $M/(\sigma_3^\perp\cap M)$, $E^{\sigma_3}_{m_{3,1}}=\lbrace 0\rbrace$. We then define the family of multifiltrations $(H^\sigma_m)$ for $\cE_{3,1}$ as follows :
\begin{itemize}
  \item For any $k<3$, any $\sigma\in\Sigma(k)$, and any $m\in M$, $H^\sigma_m=E^\sigma_m$.
  \item For $\sigma_3\in\Sigma(3)$ and $m_{3,1}\in M/(\sigma_3^\perp\cap M)$, $ H^{\sigma_3}_{m_{3,1}}=\lbrace 0\rbrace\subsetneq L=E^{\sigma_3}_{m_{3,1}}$.
  \item For $\sigma_3$ and $m\in M/(\sigma_3^\perp\cap M)$ such that $m- m_{3,1}\notin \sigma_3^\perp$, $H_m^{\sigma_3}=E_m^{\sigma_3}$.
  \item For $\sigma\in\Sigma(3)\setminus\lbrace \sigma_3\rbrace$,  $H^\sigma_m=E^\sigma_m$ for all $m\in M$. 
  \item For any $k>3$, any $\sigma\in\Sigma(k)$ and any $m\in M$, $H^\sigma_m=E^\sigma_m$ unless $\sigma_3<\sigma$ and $m\leq_{\sigma_3} m_{3,1}$ in which case $H^\sigma_m=\lbrace 0 \rbrace$.
 \end{itemize}
 Exactly as in the proof of Theorem \ref{theo:existence decomposition elementary injections}, one shows that $(H^\sigma_m)$ defines a family of multifiltrations, and that $\cE_{3,1}\to\cE_{3,0}=\cE$ is $3$-elementary. We now proceed to show that this injection is saturated, with weight $\um_\Sigma^{3,1}=-c_{\rho_0}$. We may assume that $n\geq 4$, as if $n=3$ the injection is automatically saturated. For each $\rho\notin\sigma_3(1)$, we need to compute $m_\rho^{3,1}$ as defined by \eqref{eq:definition m rho} and \eqref{eq:msigma}. Using \eqref{eq:starting sheaf multifiltration}, we find that $m_\rho^{3,1}=0$.  Then, we obtain $m_\Sigma^{3,1}=-c_{\rho_0}$ as required. Let $\sigma\in\Sigma$ such that $\sigma_3\leq \sigma$. To conclude the saturation property, we need to show that 
 $$
 W^{3,1}_\sigma= (m_{3,1}+\sigma_3^\perp)\cap (m^{3,1}_\sigma+\sigma^\vee).
 $$
 By construction, 
 $$
 W^{3,1}_\sigma=\lbrace m\in M/(\sigma^\perp\cap M)\:\vert\: \dim(E^\sigma_m)=\dim(H^\sigma_m)+1\rbrace\subset(m_{3,1}+\sigma_3^\perp)\cap (m^{3,1}_\sigma+\sigma^\vee).
 $$
 Let $m\in (m_{3,1}+\sigma_3^\perp)\cap (m^{3,1}_\sigma+\sigma^\vee)$. Then, from Lemma \ref{lem:reflexive hull},
 \begin{equation*}
 \begin{array}{ccc}
 E^\sigma_m & = &\displaystyle\bigcap_{\rho\in\sigma(1)} E^\rho (\langle m, u_\rho \rangle) \\
            & = &\displaystyle\Big( \bigcap_{\rho\in\sigma_3(1)} E^\rho (\langle m, u_\rho \rangle) \Big) \bigcap \Big( \bigcap_{\rho\in\sigma(1)\setminus\sigma_3(1)} E^\rho (\langle m, u_\rho \rangle) \Big) \\
        & = & \displaystyle E^{\sigma_3}_{m_{3,1}}\bigcap \Big(\bigcap_{\rho\in\sigma(1)\setminus\sigma_3(1)} E^\rho(\langle m, u_\rho \rangle)\Big),
 \end{array}
 \end{equation*}
 where we used that
 $m\in (m_{3,1}+\sigma_3^\perp)$. Using now $m\in  (m^{3,1}_\sigma+\sigma^\vee)$, we deduce that $E^\rho(\langle m, u_\rho \rangle)=\C^2$ for $\rho\notin\sigma_3(1)$, hence 
 $$
 E^\sigma_m=E^{\sigma_3}_{m_{3,1}}=L.
 $$
 On the other hand, as $m\in (m_{3,1}+\sigma_3^\perp)$, and from $\sigma_3\leq \sigma$, we have $H^\sigma_m\subset H^{\sigma_3}_{m_{3,1}}=\lbrace 0\rbrace$. Hence $m\in  W^{3,1}_\sigma$, and we conclude that $\cE_{3,1}\to\cE_{3,0}$ is saturated. What we used in the proof of saturation is that for $\sigma_3\leq\sigma$ and for $m\in (m_{3,1}+\sigma_3^\perp)\cap (m^{3,1}_\sigma+\sigma^\vee)$, the description $E^\sigma_m  = \displaystyle\bigcap_{\rho\in\sigma(1)} E^\rho (\langle m, u_\rho \rangle)$ from Lemma \ref{lem:reflexive hull} is valid. 
 
 We now proceed to the construction of $\cE_{3,2}\to\cE_{3,1}$, assuming $p_3\geq 2$ (otherwise we set $\cE_{4,0}=\cE_{3,1}$ and go to the next step of the construction). As parameters for this new injection, we take $(\sigma_3,m_{3,2})$ with 
 $m_{3,2}=(-c_{\rho_0},0,1)$ in the basis $(e_0^*,e_1^*,e_2^*)$ of $M/(\sigma_3)^\perp\cap M$. We consider the equivariant subsheaf $\cE_{3,2}\to \cE_{3,1}$ described by the family of multifiltrations $(G^\sigma_m)\subset (H^\sigma_m)$ with for any $\sigma\in\Sigma$ and any $m\in M$, $G^\sigma_m=H^\sigma_m$ unless $\sigma_3\leq\sigma$ and $m\leq_{\sigma_3} m_{3,2}$ in which case $G^\sigma_m=\lbrace 0 \rbrace$. As for $(H^\sigma_m)$, and following the proof of Theorem \ref{theo:existence decomposition elementary injections}, one can check that this defines indeed a family of multifiltrations and that $\cE_{3,2}\to \cE_{3,1}$ is $3$-elementary. To show saturation, we can use the same argument as for the previous injection $\cE_{3,1}\to \cE_{3,0}$. Indeed, for $\sigma_3\leq\sigma$ and for $m\in (m_{3,2}+\sigma_3^\perp)$, we have $m\notin W_\sigma^{3,1}$ and hence $H^\sigma_m  = E^\sigma_m$. Then, the description $H^\sigma_m=\displaystyle\bigcap_{\rho\in\sigma(1)} E^\rho (\langle m, u_\rho \rangle)$ from Lemma \ref{lem:reflexive hull} can be used, for any $m\in (m_{3,2}+\sigma_3^\perp)$. We conclude that $\cE_{3,2}\to \cE_{3,1}$ is saturated, and we compute $\um_\Sigma^{3,2}=-c_{\rho_0}+1$.
 
 We can iterate this argument, constructing a sequence of saturated elementary injections $\cE_{3,j}\to \cE_{3,j-1}$ with parameters $(\sigma_3,(-c_{\rho_0},0,(j-1)))$, by changing the family of multifiltrations of $\cE_{3,j-1}$ at cones $\sigma_3\leq\sigma$ and weights $m\leq_{\sigma_3} m_{3,j}$, setting the vector space to be $\lbrace 0 \rbrace$ in those cases. Once we reach $j=p_3$, we set $\cE_{4,0}=\cE_{3,p_3}$. Assuming $p_4\geq 1$ (otherwise we set $\cE_{5,0}=\cE_{4,0}$ and move on to the next step), we construct similarly $\cE_{4,1}\to\cE_{4,0}$ with parameters $(\sigma_4,m_{4,1})$ where $\sigma_4=\sigma_3+\rho_3\in\Sigma(4)$ and $m_{4,1}=(-c_{\rho_0},0,p_3,0)$ written in the basis $(e_0^*,e_1^*,e_2^*,e_3^*)$ of $M/(\sigma_4)^\perp\cap M$. To do this, we simply turn down in the family of multifiltrations for $\cE_{4,0}$ the vector space to $\lbrace 0 \rbrace$ at cones $\sigma$ containing $\sigma_4$ and weights $m\leq_{\sigma_4} m_{4,1}$. The same arguments as before show that this defines a new family of multifiltrations corresponding to a $4$-elementary saturated injection, with weight 
 $\um_\Sigma^{4,1}=-c_{\rho_0}+p_3$. The result then follows, by iterating this process.
\end{proof}


\begin{remark}
 \label{rem:higher rank existence saturated}
 The result and proof of Proposition \ref{prop:existence sequence saturated injections} can easily be adapted to higher rank equivariant torsion-free sheaves. A rank $r$ equivariant reflexive sheaf $\cF$ on a toric variety $X$ with fan $\Sigma_X$ is characterised by a family of filtrations $(F^\rho(\bullet))_{\rho\in\Sigma_X(1)}$ where for each ray $\rho\in\Sigma_X(1)$, $i\mapsto F^\rho(i)$ is an increasing and bounded filtration of some fixed vector space $F$ of dimension $r$ (see \cite{Per04}). One may, as in the rank $2$ case, normalise such a family of filtrations by tensoring by a line bundle to assume that for each $\rho$ and each $i\geq 0$, $F^\rho(i)=F$. If it is non trivial, we can also assume that for at least one $\rho_0$, $F^{\rho_0}(-1)\subsetneq F$. Assume that there exists $c_{\rho_0}\in\N $ such that
 $$
 -c_{\rho_0}:=\min\lbrace i\in\N,\: \dim(F^{\rho_0}(i))=1  \rbrace.
 $$
 To conclude, set $L:=F^{\rho_0}(-c_{\rho_0})$ and assume that the family of filtration is chosen so that for any $\rho\neq \rho_0$ and any $i\in\Z$,  $L\cap F^\rho(i)\neq \lbrace 0\rbrace$ if and only if $i\geq 0$. Such filtrations may easily be constructed. For the corresponding reflexive sheaf $\cF$, the proof of Proposition \ref{prop:existence sequence saturated injections} works fine, simply replacing $\cE$ by $\cF$ and $E$ by $F$ in the families of filtrations.
\end{remark}

Our goal is then to solve \eqref{eq:solving log} for a sequence of elementary injections as in Proposition \ref{prop:existence sequence saturated injections}. Our parameters are on one hand the $(c_\rho)_{\rho\in\Sigma(1)}$ that fix the initial Chern classes of $\cE$ and on the other hand $\up=(p_3,\ldots,p_n)$ that fix the sequence of elementary injections and the weights $\um_\Sigma^{k,j}$ as in \eqref{eq:umsigmakj}. For $3\leq k\leq n$ such that $p_k\geq 1$, and $l\leq n-k$, we introduce the notation
\begin{equation}
 \label{eq:Skl}
S_k^l:=\sum_{j=1}^{p_k} (\um_\Sigma^{k,j})^l 
\end{equation}
that will appear from \eqref{eq:log expansion saturated elementary}. Then, define the numbers $(\tilde \uc_i)_{3\leq i \leq n}$ by
\begin{equation}
 \label{eq:chern ratio expansion}
 \log\left(1+\uc_1 H+\ldots+ \uc_nH^n\right)-\log\left(1+\uc_1 H+\uc_2 H^2\right)=-\sum_{q=3}^n \tilde\uc_q \frac{H^q}{q}.
\end{equation}
Our goal is now to fix $(c_\rho)_{\rho\in\Sigma(1)}$ and $\up$ such that for each $3\leq q\leq n$,
\begin{equation}
 \label{eq:general kth equation}
 \tilde \uc_q =\sum_{k=3}^q\Big(\sum_{l=k}^q \binom{q}{l}\, A_{l,k}\, S_k^{q-l} \Big).
\end{equation}
We will meet two kinds of obstructions to solve those equations. The first one imposes arithmetic conditions on the Chern classes of $\cE$ that go beyond Schwarzenberger's constraints. The second one imposes some positivity on the combinations of the $\tilde c_q$'s and the $S_k^l$'s.

\subsection{Four dimensional constructions}
\label{sec:P4}
When $n=4$, we can  easily find solutions to those equations by hand. First, Schwarzenberger's constraints reduce to
$$
\uc_2(\uc_2+1-3\uc_1-2\uc_1^2)\equiv 0 \: \mathrm{ mod }\: 12.
$$
Then, we compute 
$$
\begin{array}{ccc}
 \tilde\uc_3 & = & -3\uc_3, \\
 \tilde \uc_4 & = & 4(\uc_1\uc_3-\uc_4).
\end{array}
$$
We also need
$$
(S_3^0,S_3^1,S_4^0)=(p_3, -(c_{\rho_0}+1)p_3+\frac{p_3(p_3+1)}{2},p_4 ).
$$
Then, the equations \eqref{eq:general kth equation} read 
$$
\left\{
\begin{array}{ccc}
2p_3 & = & \uc_3 \\
6p_4 & = & \uc_1\uc_3-\uc_4 + 9 p_3 + 6 S_3^1.
\end{array}
\right.
$$
We can substitute $p_3$ by $\frac{\uc_3}{2}$ in the second equation and we obtain the equivalent system :
\begin{equation}
 \label{eq:P4 case equations pj}
\left\{
\begin{array}{ccc}
p_3 & = & \frac{\uc_3}{2} \\
p_4 & = & \frac{\uc_1\uc_3-\uc_4}{6}+\uc_3-(c_{\rho_0}+1)\frac{\uc_3}{2}+\frac{\uc_3^2}{8}.
\end{array}
\right.
\end{equation}
As $p_i\in\N$, we obtain two type of necessary and sufficient conditions for solving this system. The first is arithmetic, to which we add Schwarzenberger's condition :
\begin{equation}
 \label{eq:P4 case arithmetic conditions}
 \left\{
\begin{array}{cccc}
\uc_3 & \equiv & 0 &\: \mathrm{ mod }\: 2\\
4(\uc_1\uc_3-\uc_4)+3\uc_3^2 & \equiv &0& \: \mathrm{ mod }\: 24\\
\uc_2(\uc_2+1-3\uc_1-2\uc_1^2) & \equiv & 0& \: \mathrm{ mod }\: 12.\\
\end{array}
\right. 
\end{equation}
The second condition is a positivity constraint :
\begin{equation}
 \label{eq:P4 case positivity}
 \frac{\uc_1\uc_3-\uc_4}{3}+\uc_3+\frac{\uc_3^2}{4}\geq c_{\rho_0}\uc_3.
\end{equation}
Remember now that $\uc_k=s_k$ is the $k$-th elementary symmetric polynomial in the $(c_\rho)_{\rho\in\Sigma(1)}$'s. Then, 
$$
\uc_1\uc_3-\uc_4 = 3\uc_4+\sum_{\rho\in\Sigma(1)} c_\rho^2\sum_{\tau\in\Sigma(2),\,\rho\notin\tau(1)} \prod_{\rho'\in\tau(1)}c_\rho'\geq 0,
$$
as for each $\rho$, $c_\rho\geq 0$. Hence we see that for fixed $c_{\rho_0}$, letting the other $c_\rho$'s go to infinity, \eqref{eq:P4 case positivity} will certainly be satisfied as the left hand side is of order $2$ in $\uc_3$ while the right hand side is of order $1$. To produce actual solutions, we may from now on set
$$
c_{\rho_0}=1.
$$
Then, \eqref{eq:P4 case positivity} is automatically satisfied, and we are left with finding solutions to \eqref{eq:P4 case arithmetic conditions}, which can be achieved easily. We will restrict to finding semistable torsion-free sheaves with positive discriminant, which forces $\cE$ to be semistable with $\Delta(\cE)>0$, hence different from a split vector bundle. From Lemma \ref{cor:loc free iff c3 vanishes}, we must consider the cases when at least $3$ terms in $(c_\rho)_{\rho\in\Sigma(1)}$ are non-zero. The simplest case is then when $(c_\rho)_{\rho\in\Sigma(1)}=(1,x,y,0,0)$ with $x,y\in\N^*$. Then, $\uc_1=1+x+y$ and $\uc_2=x+y+xy$, $\uc_3=xy$ while $\uc_4=0$. Taking $x=y=6t$ with $t\in\N^*$, \eqref{eq:P4 case arithmetic conditions} is satisfied, and solving \eqref{eq:P4 case equations pj} yields, with Proposition \ref{prop:existence sequence saturated injections}, to an explicit rank $2$ torsion-free sheaf $\cE_t$ with
$$
(\uc_1,\uc_2,\Delta)=(12t+1,12t+36t^2,24t-1).
$$
Note that those examples are stable (\ref{cor:slope stable rank two}).  Those examples produce infinitely many stable torsion-free sheaves of rank $2$ whose Chern classes satisfy the constraints of locally-free ones. For $t$ large enough, they avoid the conditions found in \cite{HoSch} that would force such a locally-free sheaf to split. Moreover, as $\Delta$ is invariant upon tensoring by a line bundle, the Chern classes of $\cE_t$ cannot be obtained from the ones of $\cE_{t'}$ by tensoring with some $\cO_{\P^4}(d)$. Finally, it follows from  Dirichlet prime number theorem that for infinitely many $t$'s, $24t-1$ is a prime number, hence cannot be divided by a square. The discriminant of a rank $2$ vector bundle is multiplied by $d^2$ when considering a pullback along a degree $d$ cover $\C\P^n\to\C\P^n$. From those observations, we see that our examples provides infinitely many Chern classes that are not related through pullbacks along finite covers nor through tensoring by a line bundle.

All the previous examples have odd first Chern class, but we can also produce infinitely many examples with $\uc_1$ even. Consider this time $(c_\rho)_{\rho\in\Sigma(1)}=(1,4t+3,4t+3,4t+3,0)$ with $t\in\N^*$. This choice is motivated by the fact that, letting $T=4t+3$, it gives
$$
\left\{
\begin{array}{ccc}
\uc_1 & = & 3T+1  \\
\uc_2 & = & 3T(T+1)\\ 
\uc_3 & = & T^2(T+3)\\
\uc_4  & = & T^3 
\end{array}
\right.
$$
hence $4(\uc_1\uc_3-\uc_4)+3\uc_3^2 =3\Big(T(T+1)(T+2)\Big)^2$. Then, the condition for solvability of \eqref{eq:P4 case arithmetic conditions} is precisely $T=4t+3$, $t\in\N$. We obtain this time infinitely many rank $2$ stable torsion-free sheaves with even first Chern class, whose Chern polynomial satisfies the known constraints for Chern polynomials of indecomposable rank $2$ locally-free  sheaves. We compute $\Delta=3T^2+6T-1=4(12t^2+24t+11)$. From a Theorem by Ricci (see \cite{Rud} and reference therein), for infinitely many values of $t$, $12t^2+24t+11$
is not divisible by a square. From this we deduce as before that those examples cannot be related by twisting by a line bundle, and that unless for degree $2$ pullbacks, may not be related by pullbacks along finite covers of the projective space. This concludes the proof of Theorem \ref{theo:intro P4 examples}.

\subsection{Five dimensional examples}
\label{sec:P5}
We turn now to $\C\P^5$, and produce, as in the previous section, an infinite family of stable rank $2$ torsion-free sheaves whose Chern classes satisfy the constraints imposed by locally freeness and indecomposability. The equations \eqref{eq:general kth equation} for $(\tilde \uc_i)$ are
\begin{equation}
 \label{eq:P5 case equations tilde ci}
\left\{
\begin{array}{ccc}
\tilde \uc_3 & = & A_{3,3}S_3^0 \\
\tilde \uc_4 & = & 4A_{3,3} S_3^1+A_{4,3}S_3^0+A_{4,4}S_4^0\\
\tilde \uc_5 & = & 10 A_{3,3}S_3^2+5 A_{4,3}S_3^1+A_{5,3}S_3^0+5A_{4,4}S_4^1+A_{5,4}S_4^0+A_{5,5}S_5^0.
\end{array}
\right.
\end{equation}
Using the $\log$-expansion, we compute 
$$
(\tilde \uc_3,\tilde\uc_4,\tilde\uc_5)=(-3\uc_3,4(\uc_1\uc_3-\uc_4),-5(\uc_5-\uc_1\uc_4-\uc_3\uc_2+\uc_3\uc_1^2)).
$$
As in the last section, we will set $c_{\rho_0}=1$, so that the terms $S_k^l$ can be expressed in terms of $(p_3,p_4,p_5)$ by using \eqref{eq:Skl} and \eqref{eq:umsigmakj}. We obtain the following :
$$
(S_3^0,S_3^1,S_3^2)=(\,p_3\,,\,\frac{1}{2}p_3^2-\frac{3}{2}p_3\, ,\, \frac{1}{3}p_3^3-\frac{3}{2}p_3^2 +\frac{13}{6}p_3\, ),
$$
$$
(S_4^0,S_4^1)=(\,p_4 \,,\, \frac{1}{2}p_4(p_4+1)+p_4(p_3-2)\,)
$$
and $S_5^0=p_5$. Recall the values of the coefficients $A_{l,k}$ from table \eqref{eq:some values of Apk}. Then, the system \eqref{eq:P5 case equations tilde ci}  can be expressed as a system of equations for the $(p_i)$'s in terms of the $\uc_i$'s :
\begin{equation}
 \label{eq:P5 case equations pi before solving}
\left\{
\begin{array}{ccc}
p_3 & = & \displaystyle\frac{\uc_3}{2} \\
p_4 & = & \displaystyle\frac{\uc_1\uc_3-\uc_4}{6}+\frac{p_3^2}{2}\\
p_5 & = & \displaystyle\frac{\uc_5-\uc_1\uc_4-\uc_3\uc_2+\uc_3\uc_1^2}{24}-\frac{1}{2}(S_3^2+3S_3^1)-\frac{5}{4}p_3+S_4^1+2p_4
\end{array}
\right.
\end{equation}
that we may solve inductively on the $p_i$'s. After  straightforward but tedious computations, we obtain the equations :
\begin{equation}
 \label{eq:P5 case equations pi after solving}
\left\{
\begin{array}{ccc}
p_3 & = & \displaystyle\frac{\uc_3}{2}, \\
& & \\
p_4 & = & \displaystyle\frac{\uc_1\uc_3-\uc_4}{6}+\frac{\uc_3^2}{8},\\
& & \\
p_5 & = & \displaystyle\frac{\uc_5-\uc_1\uc_4-\uc_3\uc_2+\uc_3\uc_1^2}{24}\\
& & \\
& & \displaystyle-\frac{\uc_3^2\uc_4}{48}-\frac{\uc_1\uc_3\uc_4}{36}-\frac{\uc_3\uc_4}{12}\displaystyle+\frac{\uc_3^3\uc_1}{48}+\frac{\uc_1^2\uc_3^2}{72}+\frac{\uc_3^2\uc_1}{12}+\frac{\uc_1\uc_3}{12}\\
& & \\
    &  &  \displaystyle+\frac{\uc_3^4}{2^7}+\frac{\uc_3^3}{24}+\frac{\uc_3^2}{16}-\frac{\uc_3}{24} \displaystyle +\frac{\uc_4^2}{72}-\frac{\uc_4}{12}.
\end{array}
\right.
\end{equation}
To produce explicit solutions, we will consider simple configurations for the $(c_\rho)$'s, and set 
$$
(c_\rho)_{\rho\in\Sigma(1)}=(1,t,t,0,0,0),\: t\in\N^*.
$$
Then, $\uc_1=2t+1$, $\uc_2=t^2+2t$, $\uc_3=t^2$ and $\uc_4=\uc_5=0$. The equations for $p_4$ and $p_5$ then become 
$$
p_4=\displaystyle\frac{\uc_1\uc_3}{6}+\frac{\uc_3^2}{8}
$$
and 
$$
p_5=\frac{-\uc_3\uc_2+\uc_3\uc_1^2}{24}+\frac{\uc_3^3\uc_1}{48}+\frac{\uc_1^2\uc_3^2}{72}+\frac{\uc_3^2\uc_1}{12}+\frac{\uc_1\uc_3}{12}+\frac{\uc_3^4}{2^7}+\frac{\uc_3^3}{24}+\frac{\uc_3^2}{16}-\frac{\uc_3}{24}.
$$
Clearly, the conditions that $p_4\geq 0$ and $p_5\geq 0$ are met for any $t\in\N^*$, and to produce solutions, one only needs to chose $t$ so that the arithmetic conditions imposed by $(p_3,p_4,p_5)\in\N^3$ are satisfied. For example, if $\uc_3=0\: \mathrm{ mod }\: 24$ we are done, and we could simply take $t\in 12\cdot\N^*$. In order to satisfy Schwarzenberger's conditions \eqref{eq:Schwarzenberger conditions}, we will further impose that $5!$ divides $\uc_2$, and for this we will take $t\in 5!\cdot \N$. To sum up, by setting 
$$
(c_\rho)_{\rho\in\Sigma(1)}=(1,120t,120t,0,0,0),\: t\in\N^*,
$$
we can solve the system \eqref{eq:P5 case equations pi after solving} and thus obtain explicit examples of equivariant and stable torsion-free sheaves whose Chern classes are satisfying the constraints of indecomposable and locally-free sheaves. They have
$$
(\uc_1,\uc_2,\Delta)=(240t+1,(120)^2t^2+240t,480t-1).
$$
We can argue as in Section \ref{sec:P4} and conclude that this family really produces an infinite family of examples that are not related by twists by line bundles nor by pullbacks along finite covers.

\subsection{Higher dimensional constructions}
\label{sec:Pn}

The previous constructions may certainly be generalised in higher dimensions. By a simple inductive argument, one can always reformulate the Equations \eqref{eq:general kth equation} as a system of equations on the $p_k$'s in terms of the $\uc_k$'s, as in \eqref{eq:P5 case equations pi after solving} :
$$
p_k=Q_k(\uc_1,\ldots,\uc_n)
$$
for $Q_k\in\Q[X_1,\ldots,X_n]$. The conditions that $p_k\in\N$ for each $k$ impose two type of constraints : positivity and integrality of $Q_k(\uc_1,\ldots,\uc_n)$, for $3\leq k \leq n$. Some numerical evidences leads us to expect that $Q_k(\uc_1,\ldots,\uc_n)\geq 0$ is always satisfied. This may be proved using that the $\uc_k$'s are symmetric elementary functions in the $(c_\rho)_{\rho\in\Sigma(1)}\in\N^{n+1}$. In any case, if one considers the special case of
$$
(c_\rho)_{\rho\in\Sigma(1)}=(1,t,t,0,\ldots,0),\: t\in\N^*,
$$
and let $t$ goes to infinity, a simple inductive argument shows that 
$$
Q_k(\uc_1,\ldots,\uc_n)\underset{t\to +\infty}{\sim}\alpha_k\, t^{\frac{2^k}{4}}
$$ with $\alpha_k >0$. Hence, for $t$ large enough, the positivity constraints of the equations are always satisfied. Then, one only needs to check the arithmetic conditions. By definition of the $\tilde \uc_k$'s, cf Equation \eqref{eq:chern ratio expansion}, as for this special case $\uc_k=0$ for $k\geq 4$, we see that all $\tilde \uc_k$ can be factored by $\uc_3$, hence by $t$. By induction again, we deduce that the same holds for the expressions $Q_k(\uc_1,\ldots,\uc_n)$. Then, for $t$ sufficiently divisible, the equations $p_k=Q_k(\uc_1,\ldots,\uc_n)$ can be solved. Moreover, if $n!$ divides $t$, then $n!$ divides $\uc_2$ and Schwarzenberger's constraints are satisfied. To sum up, we have proved :
\begin{proposition}
 \label{prop:higher dim examples}
 For each $n\geq 3$, there exists $m\in\N^*$ such that for all $t\in\N^*$, there is a stable rank $2$ equivariant torsion-free sheaf $\cE_t$ on $\C\P^n$ with Chern polynomial
 $$
 1+(2mt+1)\cdot H+(m^2t^2+2mt)\cdot H^2.
 $$
\end{proposition}
Of course, solving explicitly the equations for $(p_k)_{3\leq k\leq n}$ may become very tedious for large $n$. The above proposition raises two natural questions :
\begin{enumerate}
 \item Is it true that $Q_k(\uc_1,\ldots,\uc_n)\geq 0$ ?
 \item Can we take $m=n!$ in the statement of Proposition \ref{prop:higher dim examples} ?
\end{enumerate}
We hope to come back to those problems once we have developed the deformation theory for those sheaves.

\subsection{More prescription problems}
\label{sec:more prescriptions}
The methods and tools developed in this work can be used to solve more prescription problems. Indeed, Theorem \ref{theo:existence decomposition elementary injections} is valid for any smooth projective toric variety, and in any rank. Formula \ref{eq:log expansion saturated elementary} works in any rank on the projective space, as well as Proposition \ref{prop:existence sequence saturated injections} as long as one starts with a suitable equivariant reflexive sheaf, cf Remark \ref{rem:higher rank existence saturated}. Stability of equivariant reflexive sheaves can be checked easily in any rank as well (Proposition \ref{prop:stability equiv}). Hence, one may follow
the strategy of Section \ref{sec:prescription} and try to solve more general prescription problems for Chern polynomials of equivariant torsion-free sheaves over the projective space. As potential applications, one may try to solve the following :
\begin{enumerate}
 \item Prescribe Chern classes that satisfy locally freeness conditions in higher rank (vanishing for $i\geq r$ and the higher rank Schwarzenberger's conditions).
 \item Try to describe the set of Chern polynomials $\uc\in\Z[H]/\langle H^{n+1}\rangle$ such that the moduli space of slope stable equivariant torsion-free sheaves with Chern polynomial $\uc$ is non-empty.
\end{enumerate}

\bibliography{Ref}
\bibliographystyle{amsplain}

\end{document}